\documentclass[a4paper,11pt]{article}
\usepackage{amsmath,amssymb}
\usepackage{amsthm}
\usepackage[dvipdfmx]{graphicx}
\usepackage{float}
\usepackage{color}

\theoremstyle{defi}
\newtheorem{theo}{Theorem}[section]

\newtheorem{prop}[theo]{Proposition}
\newtheorem{lemm}[theo]{Lemma}

\newtheorem{rem}[theo]{Remark}

\newcommand{\connect}{\overset{\rm hole}{\longleftrightarrow}}
\newcommand{\bondconnect}{\overset{\rm bond}{\longleftrightarrow}}
\newcommand{\phole}{{\theta}^{\operatorname{hole}}}
\newcommand{\pbond}{{\theta}^{\operatorname{bond}}}
\newcommand{\pface}{{\theta}^{\operatorname{face}}}
\newcommand{\bond}{p_c^{\operatorname{bond}}}
\newcommand{\hole}{p_c^{\operatorname{hole}}}
\newcommand{\face}{p_c^{\operatorname{face}}}
\newcommand{\R}{\mathbb{R}}
\newcommand{\Z}{\mathbb{Z}}
\newcommand{\K}{\mathcal{K}}
\newcommand{\F}{\mathcal{F}}
\newcommand{\B}{\tilde{B}}
\newcommand{\sz}{\mbox{size}}

\title{Percolation on Homology Generators in Codimension One}
\author{Yasuaki Hiraoka\thanks{Center for Advanced Study, Institute for the Advanced Study of Human Biology (WPI-ASHBi), Kyoto University Institute for Advanced Study, Kyoto University. 
Center for Advanced Intelligence Project, RIKEN. (E-mail: hiraoka.yasuaki.6z@kyoto-u.ac.jp)}~
and 
Tatsuya Mikami\thanks{Mathematical Institute, Tohoku University (E-mail: tatsuya.mikami.s4@dc.tohoku.ac.jp)}
}

\date{}
\begin{document}
\maketitle

\tableofcontents
\newpage
\begin{abstract}
This paper introduces a new percolation model motivated from polymer materials. The mathematical model is defined over a random cubical set in the $d$-dimensional space $\R^d$ and focuses on generations and percolations of $(d-1)$-dimensional holes as higher dimensional topological objects. Here, the random cubical set is constructed by the union of unit faces in dimension $d-1$  which appear randomly and independently with probability $p$, and holes are formulated by the homology generators. 
Under this model, the upper and lower estimates of the critical probability $\hole$ of the hole percolation are shown in this paper, implying the existence of the phase transition. The uniqueness of infinite hole cluster is also proven. This result shows that, when $p > \hole$, the probability $P_p(x^* \connect y^*)$ that two points in the dual lattice $(\Z^d)^*$ belong to the same hole cluster is uniformly greater than 0.
\end{abstract}

\section{Introduction}
\subsection{Background}
\label{subsec:craze}
Percolation theory has its origin in applied problems. One of the most famous mathematical formulations is the modeling of immersion in a porous stone, which is expressed by the bond percolation model as follows. Let $\mathbb{L}^d = (\Z^d, \mathbb{E}^d)$ be the $d$-dimensional cubical lattice, where $\Z$ expresses the set of integers and $\Z^d$ and $\mathbb{E}^d$ are the sets of vertices and bonds (or edges) over the $d$-dimensional integer lattice, respectively (see Section \ref{sec:preliminaries}). For a fixed $p\in[0,1]$, each bond in $\mathbb{L}^d$ is assumed to be open randomly with probability $p$, and closed otherwise, independently of all other edges. Open bonds correspond to interstices randomly generated in the stone, and the probability $p$ means the proportion of the interstices in the stone. 

In this model,  {\it the percolation probability} $\pbond(p) = P_p(|C(0)| = \infty)$ has been extensively studied. Here, $P_p$ expresses the probability measure constructed as the product measure of those from all bonds, and $C(0) \subset \mathbb{L}^d$ denotes the connected component containing the origin in the subgraph which consists of all open bonds. The percolation probability $\pbond(p)$ increases as the probability $p$ increases, and it has been of great interest in {\it the critical probability} $\bond(d) = \inf \{p : \pbond(p) > 0\}$. 

For $d \geq 2$, it is easy to show that $0 < \bond(d) < 1$. This implies that the bond percolation model possesses two phases $p > \bond(d)$ and $p < \bond(d)$ called {\it supercritical} and {\it subcritical} phases, respectively, and the phase transition occurs at the critical probability $\bond(d)$. Namely, open bond clusters can be infinitely large in the supercritical phase, while they are always in finite size in subcritical phase. One of the most remarkable properties showing the phase transition is formulated as follows.
\begin{theo}
\rm
If $p > \bond$, then there exists $c := c(p) > 0$ such that
\begin{align}
P_p(x \overset{\text{bond}} \longleftrightarrow y) \geq c  \ \text{for any } x, y  \in \Z^d.
\label{supercritical}
\end{align} 
If $p < \bond$, then there exists $\sigma := \sigma(p) > 0$ such that
\begin{align}
P_p(x \overset{\text{bond}} \longleftrightarrow y) \leq e^{-\sigma \|x - y\|_1}  \ \text{for any } x, y  \in \Z^d.
\label{subcritical}
\end{align}
\label{theorem:phase_transition}
\end{theo}
Here, we denote by $x \bondconnect y$ the event that two vertices $x, y$ are connected by some open paths. Theorem \ref{theorem:phase_transition} shows that the probability $P_p(x \bondconnect y)$ behaves differently between the two phases. We refer to \cite{Grimmett} for more details about the bond percolation. 

Recently, a new type of percolation phenomenon is pointed out in the study of polymer materials \cite{craze}. In that paper, they study the generating mechanism of craze formations appearing in the uniaxial deformation of polymers (Kremer-Grest model) by molecular dynamics simulations. 
Then, they found by applying persistent homology that a large void corresponding to a craze of the polymer starts to appear by the process of coalescence of many small voids. Namely, this paper suggests that ``percolation of nanovoids" is the key mechanism to initiate craze formations, comparing to the other possibilities such as direct growing of some selected small voids.

On the other hand, higher dimensional models defined over the cubical lattice $\mathbb{L}^d = (\Z^d, \mathbb{E}^d)$ have also been studied recently in random topology \cite{tsunoda,ww}. In their model, $k$-dimensional elementary cubes (a product of $k$ intervals with length one) are assumed to be open with  probability $p$. From the construction, it naturally includes the bond percolation model mentioned above.
Then, some topological properties of the resulting random cubical set are studied and, in particular, the paper \cite{tsunoda} shows several limit theorems on higher dimensional homology of the random cubical set. These results are regarded as higher dimensional generalizations of the classical studies on connected components in random graphs \cite{ErdosRenyi}, which correspond to $0$-dimensional topological objects. 

We note that the bond percolation model explained above (and most percolation models studied in probability theory so far) focuses on the infinite clusters of connected components. 
However, in view of the recent progress of random topology (e.g., \cite{bk,kahle}), it is natural to consider a new type of percolation model which directly deals with higher dimensional topological objects. 

In this paper, we introduce a higher dimensional percolation model, called hole percolation, motivated from the craze formation of polymer materials. 
While the classical bond percolation theory mainly studies clusters of vertices (i.e., 0-dimensional objects), our model focuses on  clusters of holes as higher dimensional topological objects. More precisely,  we use homology generators in codimension one for representing the holes, and then study infinite clusters of those holes, which model the percolation of nanovoids in polymer materials. 

Historically, the paper \cite{plaquette} uses the plaquette percolation model, which is almost equivalent to the  models studied in \cite{tsunoda,ww}. 
Their interest in that paper is to study the percolation problem which also allows entanglement to the usual bond percolation, and the plaquette model is introduced as a subsidiary tool to study entanglement. 
We remark that, although our hole percolation model is constructed based on the setting in \cite{tsunoda}, our interest is infinite clusters of holes, and hence is different from 
entanglements.

\subsection{Main results}
Our mathematical model is briefly explained as follows (see Section \ref{sec:model} for details). In the $d$-dimensional cubical lattice $\mathbb{L}^d$, we assume that each unit cube in dimension $d-1$ called face is open with probability $p$ and closed otherwise, independently of all other faces. For a configuration $\omega$ of faces, we focus on the homology  in codimension 1 of its realization  $K(\omega)$, i.e., $H_{d-1}(K(\omega))$. Each generator of $H_{d-1}(K(\omega))$ corresponds to a bounded component of $\R^d \setminus K(\omega)$, which we call ``hole", and we study the percolation of holes.  

To that aim, we define the so called hole graph, that is, the vertices consist of holes and the edges are assigned for adjacent holes. Then, in the same way as the bond percolation model, we define the percolation probability $\phole(p) := P_p(|G_{0^*}(\omega)| = \infty)$ and the critical probability $\hole := \inf \{ p \in [0,1] : \phole(p) > 0 \}$, where $G_{0^*}(\omega)$ is a fixed connected component of the hole graph. We call this model {\it the hole percolation model} in this paper. 

Under this setting, we first give estimates of the critical probability $\hole$ and, in particular, we show that $0 < \hole < 1$ (Theorem \ref{theo:pc_estimate}). This implies that there exists two phases even in the hole percolation model. To find an upper bound of $\hole$, we use the dual lattice $(\mathbb{L}^d)^*$, which is obtained by shifting $\mathbb{L}^d$ to the vector $(1/2, \ldots , 1/2)$. There is a natural bijective correspondence between faces in $\R^d$ and dual bonds transversely intersecting each other. Under this bijection, we assume that each dual bond is open if and only if the corresponding face is closed, leading to the bond percolation model in $(\mathbb{L}^d)^*$ with probability $1-p$. Then, it can be shown that the holes in $\R^d$ correspond to the finite clusters in $(\mathbb{L}^d)^*$. Under this relation, the generation of holes is studied via finite clusters in the dual bond percolation. 

Moreover, we show the analogues of the estimate (\ref{supercritical}) of the probability $P_p(x \overset{\text{bond}} \longleftrightarrow y)$ in the supercritical phase. For the bond percolation model, the uniqueness of the infinite cluster plays an important role to prove the estimate (\ref{supercritical}). Namely, if two vertices belong to infinite clusters, then those two vertices are connected by an open path in the unique infinite cluster, and thus, the probability $P_p(x \overset{\text{bond}} \longleftrightarrow y)$ is bounded below, independently on the distance of $x, y$. 
Following this strategy, 
we show the uniqueness of the infinite cluster in the hole percolation model (Theorem \ref{theo:hole_unique}), and prove the analogues statement in Theorem \ref{theo:connection_in_supercritical}. 

We also discuss differences between the bond and hole percolation models. A significant difference, which makes difficult the analysis of shapes and sizes of hole graphs, is that the generation of holes cannot be decided in the bounded area. We observe how this difficulty influences properties of the hole percolation model.

The paper is organized as follows. In Section 2, we introduce the setting of the hole graph and show the main theorems. In Section 3 and Section 4, we prove two main theorems: the estimate of the critical probability and the uniqueness of infinite cluster, respectively. In Section 5, we state the other properties of the hole percolation model and carefully discuss the difference between the hole percolation model and the classical one.

\section{Model and main theorems}
\label{sec:model}
\subsection{Preliminaries}\label{sec:preliminaries}
We denote by $\| \cdot \|_p$ the $L_p$-norm, and by $|G|$ the number of vertices of a graph $G$. Assume $d \geq 1$. Let $\Z^d$ be the set of all vectors $x = (x_1, x_2, \ldots , x_d)$ with integer coordinates, and we define
\begin{align*}
\mathbb{E}^d = \{ \big<x, y \big> : x,y \in \Z^d, \, \| x -y \|_1 = 1\}
\end{align*}
as the set of edges. We call the pair $\mathbb{L}^d = (\Z^d, \mathbb{E}^d)$ the {\it d-dimensional cubical lattice}. We define the sample space $\Omega := \{0, 1\}^{\mathbb{E}^d}$ and the $\sigma$-field $\F$ of $\Omega$ generated by finite-dimensional cylinder sets\footnote{A finite-dimensional cylinder set is a set $\{\omega \in \Omega: \omega_{e_i} = \epsilon_i, i = 1,2, \ldots ,n\}$ for some $n \in \mathbb{N}$, $e_1, \ldots , e_n \in {\mathbb{E}^d}$ and $\epsilon_i \in \{0, 1\}$. }. For $p \in [0,1]$, define the probability measure $P_p$ on $(\Omega, \F)$ as the product measure $\Pi_{e \in {\mathbb{E}^d}} \mu_e$, where $\mu_e$ is the measure on $\{0, 1\}$ such that $\mu_e(1) = p$. We denote by $E_p(\cdot)$ the expectation with respect to $P_p$. For a sample $\omega = (\omega_e \, : \, e \in \mathbb{E}^d) \in \Omega$, called a {\it configuration}, we say a bond $e \in \mathbb{E}^d$  is {\it open} (resp. {\it closed}) if $\omega_e = 1$ (resp. 0). \par
For a configuration $\omega$, let $K(\omega) \subset \mathbb{L}^d$ be a subgraph which consists of $\Z^d$ and all open bonds in $\omega$. We denote by $C(x)$ the cluster at $x$, i.e., the connected component of $K(\omega)$ containing the vertex $x$, and we write $C(0)$ the cluster at the origin. For the number $|C(0)|$ of vertices, which is a random variable,  we define the {\it percolation probability} as
\begin{align*}
\pbond(p) = P_p(|C(0)| = \infty).
\end{align*}
We also define the {\it critical probability} as
\begin{align*}
\bond(d) = \inf \{p : \pbond(p) > 0\},
\end{align*}
which is the critical point of $p$ for which $\pbond(p) > 0$.
It is one of the great interests of percolation theory to find or estimate $\bond(d)$.
\begin{rem}
\label{rem:Kesten}
\rm
We can easily see $\bond(1) = 1$. For $d=2$, Harris \cite{Harris} proved that $\pbond(1/2) = 0$, and Kesten \cite{Kesten} proved that $\bond(2) = 1/2$ . 
\end{rem}

\begin{rem}
\rm
\label{bond_dimension}
For any dimension $d \geq 1$, we can easily check
\begin{align*}
\bond(d+1) \leq \bond(d). 
\end{align*}
Indeed, by embedding $\mathbb{L}^d$ into $\mathbb{L}^{d+1}$ in a natural way as the projection of $\mathbb{L}^{d+1}$ onto the subspace generated by the first $d$ coordinates, an infinite cluster at the origin in $\mathbb{L}^d$ can be regarded as one in $\mathbb{L}^{d+1}$. Hence, together with Remark \ref{rem:Kesten}, we have an upper bound $\bond(d) \leq 1/2$ for $d \geq 2$. For a lower bound, we use \cite[Theorem 1.33]{Grimmett} and obtain
\begin{align*}
\frac{1}{2d-1} \leq \bond(d)
\end{align*}
for $d \geq 1$. It follows from these inequalities that $0 < \bond(d) < 1$ for $d \geq 2$, which implies that there are two phase (supercritical $p > \bond$ and subcritical $p < \bond$) in the bond percolation model.
\end{rem}
\par
We denote by $x \overset{\text{bond}} \longleftrightarrow y$ the statement that two vertices $x, y \in \Z^d$ belong to the same cluster, and by $N^{\text{bond}}_{\infty}$ the number of infinite clusters of $K(\omega)$. For $x \in \Z^d$, let $\tau_x: \Omega \longrightarrow \Omega$ be the transformation $\tau_x \omega_e = \omega_{e+x}$ ($e \in \mathbb{E}^d$). Then $\tau_x$ is measure preserving on $(\Omega, \F, P_p)$ and $(\Omega, \F, P_p, \tau_x)$ is ergodic.
Since the event $\{N^{\text{bond}}_{\infty} = k \}$ is translation-invariant for $k \in \mathbb{N} \cup \{ \infty \}$, i.e., $\tau_x\{N^{\text{bond}}_{\infty} = k \} = \{N^{\text{bond}}_{\infty} = k \}$, $P_p(N^{\text{bond}}_{\infty} = k)$ is equal to either $0$ or $1$. Naturally, the value of $k$ with $P_p(N^{\text{bond}}_{\infty} = k) = 1$ depends on the choice of $p$. Clearly $k = 0$ in the subcritical phase.  Burton and Keane \cite{uniqueness} showed that $k$ is equal to 1 when $p$ satisfies $\pbond(p) > 0$. 
\begin{theo}[Burton and Keane \cite{uniqueness}]
\label{th:uniqueness}
\rm
If $\pbond(p) > 0$, then $N^{\text{bond}}_{\infty} = 1$ almost surely
\end{theo}
\par
We remark that this theorem includes the statement that $\pbond(\bond) > 0$ implies $N^{\text{bond}}_{\infty} = 1$  almost surely, though the positivity of $\pbond(\bond)$ is not shown.
Next, we review the FKG inequality (\cite[Theorem 2.4]{Grimmett}), which plays an important role in percolation theory. For the purpose of applying to our model introduced in the next section, we formulate this theorem in a slightly more general setting than \cite[Theorem 2.4]{Grimmett}, yet its proof is similar. \par
For  the general setting, we replace $\mathbb{E}^d$ with an at most countable set $S$ and we consider the product space $(\Omega, \F, P_p)$ similarly defined over $S$. Then, there is a natural partial order on $\Omega$, given by $\omega \leq \omega^{\prime}$ if and only if $\omega_s \leq \omega^{\prime}_s$ for all $s \in S$. A random variable $X$ on $(\Omega, \F)$ is called {\it increasing} if $X(\omega) \leq X(\omega^{\prime})$ whenever $\omega \leq \omega^{\prime}$, and an event $A$ is called {\it increasing} if its indicator function $I_A$ is increasing.
\begin{rem}
\rm
The event $A \in \F$ is increasing if and only if both $\omega \leq \omega^{\prime}$ and $\omega \in A$ imply $\omega^{\prime} \in A$.
\end{rem}
The FKG inequality is expressed as follows.
\begin{theo}[FKG inequality]
\rm
\label{FKG}
If $X$ and $Y$ are increasing random variables on $(\Omega, \F, P_p)$ such that $E_p(X^2) < \infty$ and $E_p(Y^2) < \infty$, then
\begin{align*}
E_p(XY) \geq E_p(X)E_p(Y).
\end{align*}
\end{theo}

\begin{rem}
\rm
If $A, B \in \F$ are  increasing events, then we may apply the FKG inequality to their indicator functions $I_A$ and $I_B$ to find that   
\begin{align*}
P_p(A \cap B) \geq P_p(A)P_p(B).
\end{align*}
\end{rem}
Theorem \ref{th:uniqueness} and the FKG inequality imply the estimate (\ref{supercritical}). 

\begin{proof}[Proof of {\rm (\ref{supercritical})}]
If $p > \bond$, then there exists the unique infinite cluster almost surely and we obtain
\begin{align*}
P_p(x \overset{\text{bond}} \longleftrightarrow y) &\geq P_p(|C(x)| = \infty,\,  |C(y)| = \infty).
\end{align*}
We may apply the FKG inequality to the increasing events $\{ |C(x)| = \infty \}, \{ |C(y)| = \infty \}$ to find that the right hand side is bounded below by
\begin{align*}
P_p(|C(x)| = \infty)P_p(|C(y)| = \infty) = \pbond(p)^2 > 0, 
\end{align*}
which does not depend on $x, y  \in \Z^d$.
\end{proof}

Throughout this paper, we use the following notations. The {\it $d$-dimensional dual lattice} $(\mathbb{L}^d)^*$ is the lattice obtained by translating the d-dimensional cubical lattice by the vector $(1/2, \ldots, 1/2)$, that is, the pair $(\mathbb{L}^d)^* = ((\mathbb{Z}^d)^*, (\mathbb{E}^d)^*)$ of $(\mathbb{Z}^d)^*  := \{x^* =  x + (1/2, \ldots ,1/2) :  x \in \Z^d \}$ and $(\mathbb{E}^d)^* := \{ \big< x^* , y^*\big> :  \|x^*-y^*\|_1 = 1,  \, x^*,y^* \in (\mathbb{Z}^d)^*\}$. \par
For $n \in \Z_{\geq 0}$, let $B(n)$ be the box $\{x \in \Z^d : \|x\|_{\infty} \leq n\}$ and $\B(n)$ be $\{x^* \in (\Z^d)^* : \|x^*\|_{\infty} < n\}$. For a subset $V \subset \Z^d$, the boundary of $V$, denoted by $\partial V$, is the set of vertices in $V$ which is adjacent to some vertices in $\Z^d \setminus V$. An edge $e =  \big< x , y\big> \in \mathbb{E}^d$ is called a boundary edge of $H \subset \mathbb{L}^d$ if either $x$ or $y$ is the vertex of a subgraph $H$. For a subgraph $S \subset \mathbb{L}^d$, we denote by $\Delta S$ the set of all boundary edges of $S$. \par

\subsection{Hole graph}
In this and next subsection, we introduce our model which is a higher dimensional generalization of the usual percolation models. First, in this subsection, we define a {\it hole graph}, which corresponds to a chain of nanovoids in the craze formation. \par
Here we briefly review the concept of cubical set, which is used for defining our percolation model. We refer to \cite{cubical} for more details．An {\it elementary interval } is a closed interval $I \subset \mathbb{R}$ of the form $I = [l, l+1]$ or $I = [l, l]$ for some $l \in \mathbb{Z}$. An elementary interval $I$ is said to be {\it nondegenerate} (resp. {\it degenerate)} if $I = [l, l+1]$\ (resp. $I = [l,l]$). An {\it elementary cube} in $\R^d$ is a product $Q = I_1 \times I_2 \times \cdots \times I_d$ of elementary intervals, and the dimension of $Q$ is defined as
\begin{align*}
\operatorname{dim} Q := \# \{1 \leq i \leq d : I_i \mbox{ is nondegenerate}  \}.
\end{align*}
Denote by $\mathcal K^d_k$ the set of all elementary cubes in $\R^d$ with dimension $k$. $X \subset \R^d$ is called a {\it cubical set} if $X$ can be written as a union of elementary cubes. Note that  an infinite union of elementary cubes is also included in our definition of cubical sets although it is not in \cite{cubical}. The dimension  $\operatorname{dim }X$ of $X$ is defined as
\begin{align*}
\operatorname{dim }X := \max \{ \operatorname{dim} Q : Q \subset X \}.
\end{align*}
\par
We now introduce hole graphs. In this paper, a {\it face} in $\R^d$ means an elementary cube with dimension $d-1$. A hole graph is constructed from a cubical set consisting of faces. Given a cubical set $X$ with dimension $d-1$, a finite graph $G^n(X)$ is first constructed by restricting to the $n$-window $\Lambda^n:= [-n,n]^d \subset \R^d$ in the following way. Let 
\begin{align*}
\R^d \setminus (X \cap \Lambda^n) = D_0 \sqcup D_1 \sqcup \cdots \sqcup D_{\beta^n} 
\end{align*}
be the unique decomposition of the complement $\R^d \setminus (X \cap \Lambda^n)$, where $D_0$ is an unbounded connected domain and $D_i$ is a bounded connected domain for each $i =1,2, \ldots ,\beta^n$.
\begin{rem}
\label{Betti}
\rm
$\beta^n$ is the $(d-1)$-th Betti number of $X \cap \Lambda^n$. There is a natural bijective correspondence between the generators of the homology group $H_{d-1}(X \cap \Lambda^n) \simeq \mathbb{K}^{\beta^n}$ of $X \cap \Lambda^n$ in dimension $d-1$ and the bounded connected components $D_i \, (i = 1,2, \ldots ,\beta^n)$. 
\end{rem}
We call each connected domain $D_i \, (i = 1,2, \ldots ,\beta^n)$ a {\it hole}. Then the graph $G^n(X)$ is defined as follows. Its vertex set is  the set of holes, and two vertices are adjacent if and only if they share common boundary faces. That is, for two holes $D, D^{\prime}$, we define $D \backsim D^{\prime}$ if there exists a face $Q \in \K^d_{d-1}$ such that $Q$ is in the boundary of $D, D^{\prime}$. The graph $G^n(X)$ defined above clearly increases with the radius $n$. We define the hole graph of the cubical set $X$ as the limit $G(X) := \bigcup_{n \in \mathbb{N}}G^n(X)$. 
\par
\begin{figure}[H]
  \centering
  \includegraphics[width=6cm]{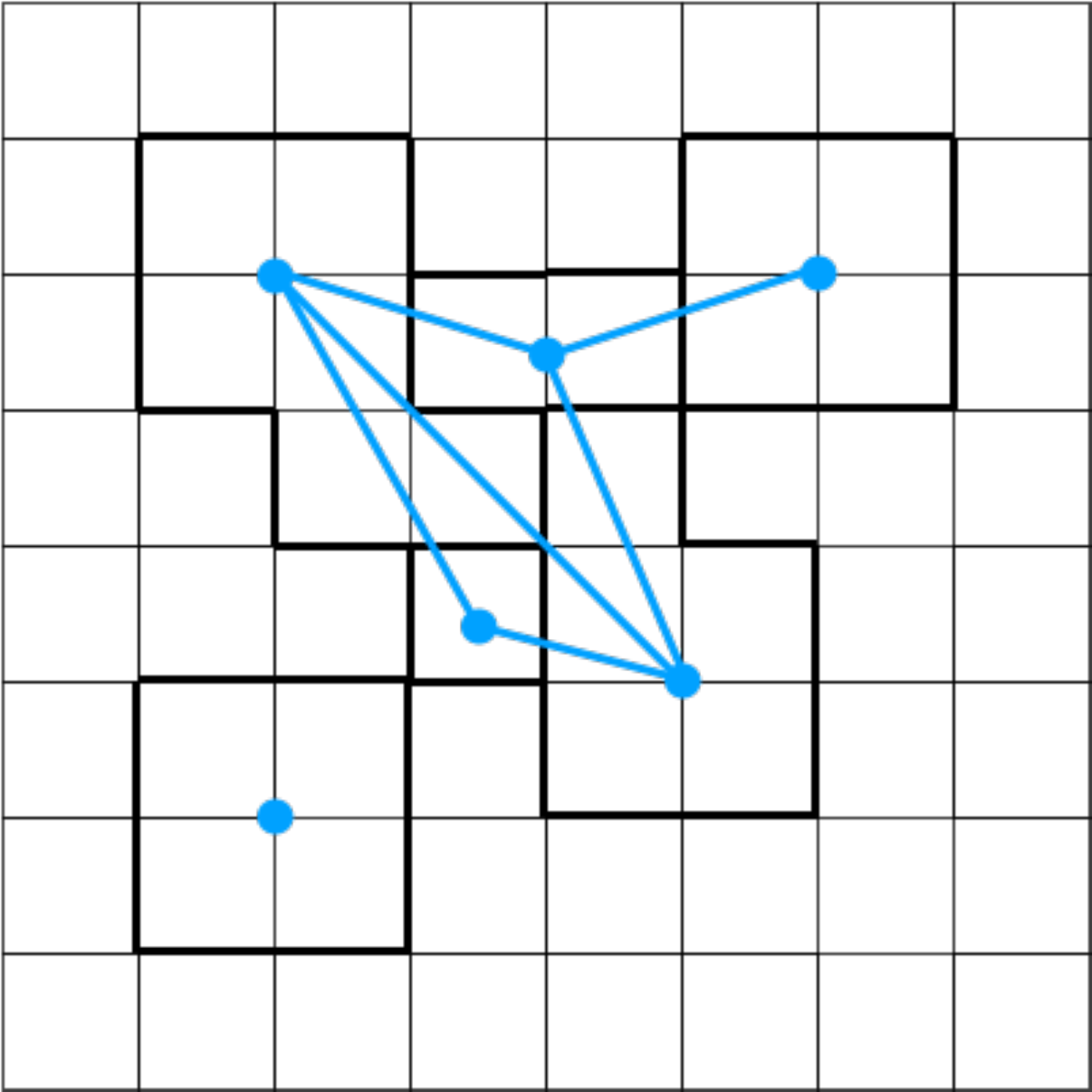}
  \caption{$d=2$. The cubical set $X$ (black) and the induced hole graph (blue)}
\end{figure}
Note that we often think of a hole graph as an embedded figure into $\R^d$, though the hole graph itself is an abstract graph induced by a cubical set. To detect the location of holes, we make use of the dual vertices. In this paper, we sometimes regard a hole $D$ as a subset of $(\Z^d)^*$, that is, 
\begin{align}
\label{regarding_subgraph}
D = \{ x^* \in (\Z^d)^* : x^* \in D \} \subset (\Z^d)^*.
\end{align}
We denote by $G_{x^*}(X)$ the connected component of the graph $G(X)$ containing the hole $D$ with $x^* \in D$. If there is no such $D$, we set $G_{x^*}(X) = \emptyset$. For $x^*, y^* \in (\Z^d)^*$, we write $x^* \connect y^*$ if $G_{x^*} = G_{y^*}$, that is, $x^*$ and $y^*$ are connected by a hole path. 

\subsection{Face percolation}
\label{subsec:Face_percolation}
In this subsection, we introduce {\it the face percolation model} which will be used for representing random generations of holes. Let $d \geq 2$.  As a sample space, we take $\Omega := \{0, 1\}^{\K^d_{d-1}}$ and $\F$ to be the $\sigma \mathchar`-$field of subsets of $\Omega$ generated by finite dimensional cylinder sets. The probability measure $P_p$ is the product measure $\Pi_{Q \in {\K^d_{d-1}}} \mu_Q$, where $\mu_Q$ is the measure on $\{0, 1\}$, given by $\mu_Q(\{1\}) = p$.  We say that $Q$ and $Q^{\prime}$ are adjacent if $Q \cap Q^{\prime} \in \K^d_{d-2}$. We denote the ``origin" of $\mathcal K_{d-1}^d$ by $Q_0 :=[0,0] \times [0,1] \times \cdots \times [0,1]$. Similar to the ordinary bond percolation model, we define the percolation probability and the critical probability as
\begin{align*}
&\pface(p) := P_p(| C(Q_0) | = \infty), \\
&\face(d):= \inf \{p : \pface(p) > 0  \},
\end{align*}
respectively, where $C(Q)$ denotes the connected component of faces including $Q$ and $| \cdot |$ denotes the number of faces.
\begin{rem}
\rm
\label{rem:2d}
For $d=2$, a face simply means a bond, implying $\face(2) = \bond(2) = 1/2$ from Remark \ref{rem:Kesten}.
\end{rem}
In this paper, we estimate the critical probability of the face percolation model in $\R^d$ for $d \geq 2$ as follows.
\begin{prop}
\rm
\label{face_lower}
For $d \geq 2$, it holds that
\begin{align*}
1/(6d - 7) \leq \face(d).
\end{align*}
\end{prop}
\begin{proof}
We first show that the number of faces adjacent to one face is equal to $6(d-1)$. Fix an arbitrary face $Q = I_1 \times I_2 \times \cdots \times I_d$. Without loss of generality, we can assume that only $I_1$ is degenerate. For a face $Q^{\prime} =  I_1^{\prime} \times I_2^{\prime} \times \cdots \times I_d^{\prime}$ adjacent to $Q$, we see
\begin{align}
Q \cap Q^{\prime} = (I_1 \cap I_1^{\prime}) \times (I_2 \cap I_2^{\prime}) \times \cdots \times (I_d \cap I_d^{\prime}) \in \K^d_{d-2}.
\label{adjacent}
\end{align}
We count the number of possible $Q^{\prime}$ as follows. 
\begin{itemize}
\item Suppose that $I_1^{\prime}$ is degenerate. It follows from (\ref{adjacent}) that $I_1 = I_1^{\prime}$ and only one of $(I_2 \cap I_2^{\prime}), \ldots ,(I_d \cap I_d^{\prime})$ is degenerate. For the degenerate $(I_i \cap I_i^{\prime})$, $I_i^{\prime}$ must intersect with either left or right end of $I_i$, since both $I_i$ and $I_i^{\prime}$ are nondegenerate. For other coordinates, we clearly see that $I_j = I_j^{\prime}$. Therefore, we have $(d-1) \times 2 = 2(d-1)$ different possible $Q^{\prime}$. \par
\item Suppose that $I_i^{\prime}$ is degenerate for some $i \neq 1$. Then, only $(I_1 \cap I_1^{\prime})$ and $(I_i \cap I_i^{\prime})$ are the degenerate intervals of $Q \cap Q^{\prime}$. From the same discussion as above, each $I_1^{\prime}$ and $I_i^{\prime}$ has 2 combinations, and the other coordinates coincide. Therefore, we have $(d-1)\times 2 \times 2 = 4(d-1)$ different possible $Q^{\prime}$. 
\end{itemize}
Thus the number of possible adjacent $Q^{\prime}$ is $6(d-1)$. Let us denote by $\sigma(n)$ the number of paths with $n$ faces from the origin. Then $\sigma(n)$ is bounded above by $6(d-1) \times [6(d-1)-1]^{n-2}$. Let $N(n)$ be the random variable counting the number of open paths from $Q_0$ which consist of greater than or equal to $n$ faces. For $p < 1/(6d-7)$, we can estimate
\begin{align*}
\pface(p) &\leq P_p(N(n) \geq 1) \\
               &\leq E_p(N(n)) \\
               & = p^n \sigma(n) \\
               & \leq 6(d-1)p^2 \times [6(d-7)p]^{n-2} \longrightarrow 0
\end{align*}
 as $n \longrightarrow \infty$. Thus we obtain $\pface(p) = 0$, which completes the proof of $1/(6d-7) \leq \face(d)$.
\end{proof}

\begin{prop}
\rm
\label{face_upper}
For $d \geq 2$, it holds that
\begin{align*}
\face(d+1) \leq \face(d).
\end{align*}
\end{prop}
\begin{proof}
Consider face percolation on $\R^{d+1}$ with probability $p \in [0,1]$, and define a face $Q =  I_1 \times I_2 \times \cdots \times I_d$ in $\R^d$ to be open if and only if the face $\tilde{Q} := I_1 \times I_2 \times \cdots \times I_d \times [0,1]$ in $\R^{d+1}$ is open. This induces face percolation on $\R^d$ with probability $p$. In this situation, if the two faces  $Q, Q^{\prime}$ in $\R^d$ are adjacent, so are the corresponding faces $\tilde{Q}, \tilde{Q}$ in $\R^{d+1}$ (see Figure \ref{fig:face}). Indeed, $Q \cap Q^{\prime} \in \K^d_{d-2}$ implies $\tilde{Q} \cap \tilde{Q}^{\prime} = Q \cap Q^{\prime} \times [0, 1] \in \K^{d+1}_{d-1}$. Thus $\face(d+1) \leq \face(d)$.

\begin{figure}[H]
  \centering
  \includegraphics[width=6cm]{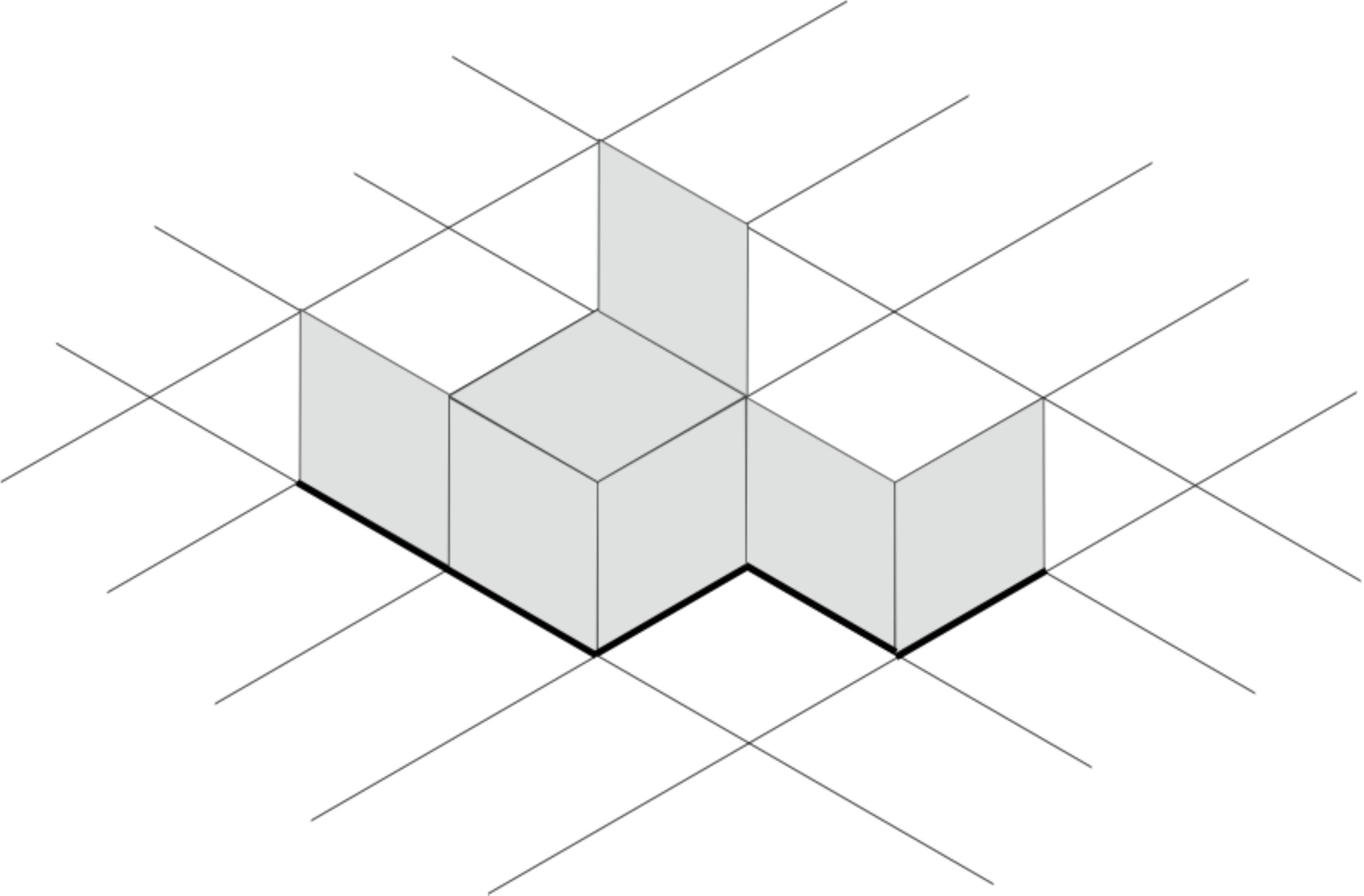}
  \caption{$d=2$. Faces on $\R^3$ (gray) and on $\R^2$ (thick)}
   \label{fig:face}
\end{figure}
\end{proof}
From Remark \ref{rem:2d} and Proposition \ref{face_upper}, we obtain the estimate $\face(d) \leq 1/2$ for any $d \geq 2$. 
\par
We now construct the hole graph over face percolation model. For a configuration $\omega \in \Omega$, we denote by $K(\omega)$ its realization into $\R^d$, i.e, a cubical set $K(\omega) := \bigcup_{Q \text{:open}} Q$ with dimension $d-1$. We simply denote $G(K(\omega))$ and $G^n(K(\omega))$ by $G(\omega)$, $G^n(\omega)$, respectively. We set the origin $0^* = (1/2, \ldots ,1/2)$ and define the percolation probability and the critical probability as
\begin{align*}
&\phole(p) := P_p(|G_{0^*}(\omega)| = \infty),  \\
&\hole := \inf \{p : \phole(p) > 0 \},
\end{align*}
respectively. We call this model {\it the hole percolation model}.
\begin{rem}
\rm
\label{increasing_event}
Note that the event $\{|G_{0^*}(\omega)| = \infty \}$ is increasing. Indeed, when we add an open face $Q$ to a configuration $\omega$, we easily see that the possible changes of the induced hole graph is given as follows (see Figure \ref{increasing}):
\begin{itemize}
\item[(a)] a vertex is  divided into two vertices and a connecting edge, or
\item[(b)] a new vertex is generated.
\end{itemize}

\begin{figure}[H] 
  \centering
  \includegraphics[width=8cm]{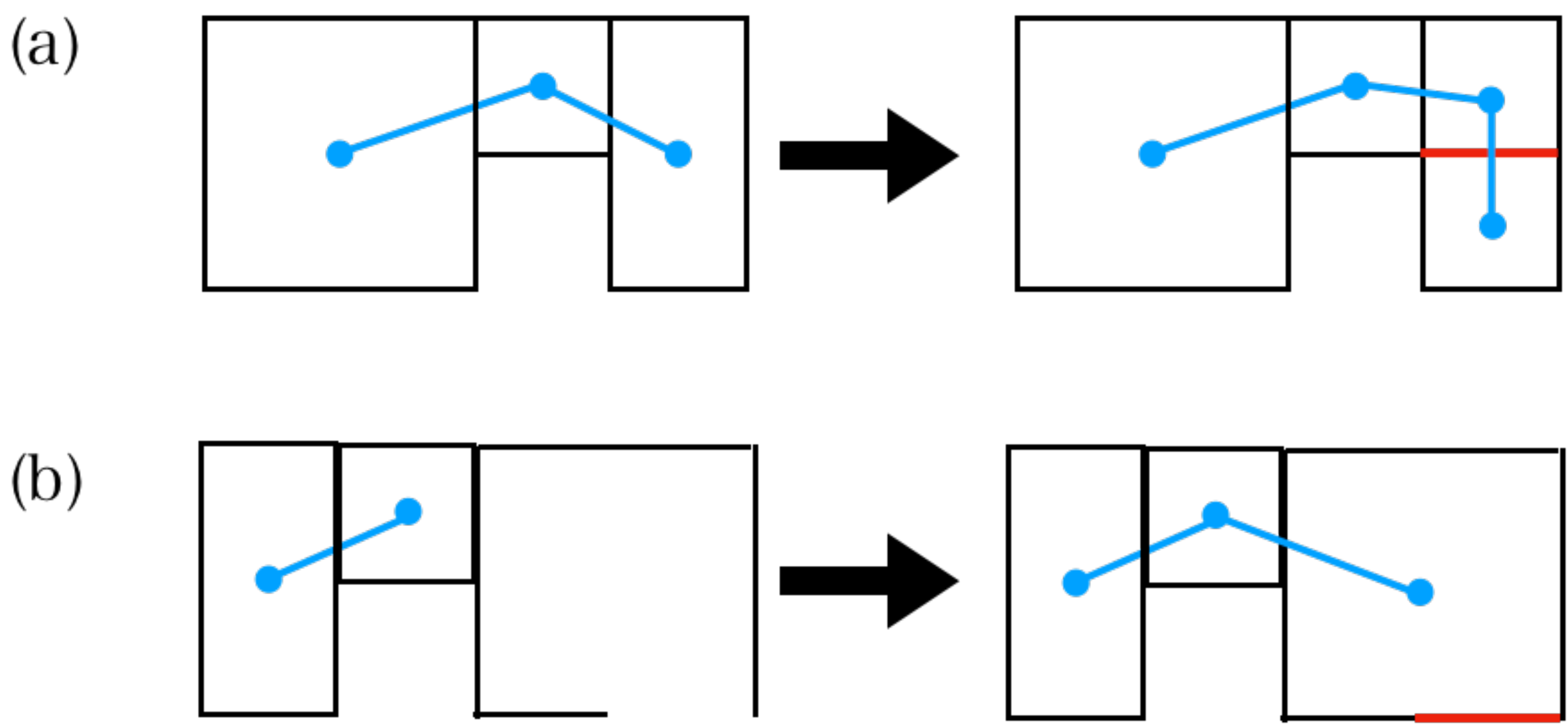} 
  \caption{the case of (a) and (b)}
  \label{increasing}
\end{figure}  

Note that these cases cannot obstruct the event $\{ G_{0^*}(\omega) = \infty \}$ though the graph structure is changed. From this, we can see that  the percolation probability $\phole(p)$ defined above is an increasing function of $p$ (see \cite[Theorem 2.1]{Grimmett} for the proof).
Thus, for the critical probability $\hole$, we have the following relation:
\begin{equation*}
\phole(p) \begin{cases}
 = 0, & \text{if } p < \hole, \\
 > 0, & \text{if } p > \hole.
\end{cases}
\end{equation*}
\end{rem}

From now on, we will denote by $(\Omega, \F, P_p)$ the probability space of the face percolation model with probability $p$. 

\subsection{Main theorems}
In this subsection, we show the main theorems. First, we estimate the critical probability $\hole$ of the hole percolation model as follows.
\begin{theo}\label{theo:pc_estimate}
\rm
\label{upper}
For any dimension $d \geq 2$, it holds that
\begin{align*}
\face(d) \leq \hole(d) \leq 1 - \bond(d).
\end{align*}
\end{theo}

\begin{rem}
\rm
\label{exact_1/2}
It immediately follows from the Proposition \ref{face_lower} that $\hole$ is in the open interval $(0,1)$.
For $d=2$, it follows from Remark \ref{rem:2d} that
\begin{align*}
1/2 \leq \hole(2) \leq 1 - 1/2 = 1/2,
\end{align*}
hence $\hole(2) = 1/2$.
\end{rem}
The other main theorem we prove is an analogue of Theorem \ref{th:uniqueness}.
\begin{theo}\label{theo:hole_unique}
\rm
\label{unique}
Suppose $d \geq 2$. Let $N_{\infty}$ be the random variable which counts the number of infinite hole clusters, and suppose $\phole(p) > 0$. Then, $N_{\infty} = 1$ almost surely.
\end{theo} 
\begin{rem}
\rm
\label{ergodic}
Similarly to the case of the bond percolation model, it is clear that $(\Omega, \F, P_p, \tau_x)$ is ergodic. Since the event $\{ N_{\infty} = k\}$ is translation invariant for $k \in \mathbb{N} \cup \{ \infty \}$, there exists $k \in \mathbb{N} \cup \{\infty\}$, depending only on $p$, such that $P_p(N_{\infty} = k ) = 1$. Theorem \ref{unique} states that $k$ must be 0 or 1, which implies
\begin{equation*}
N_{\infty} = \begin{cases}
0, & \text{if } p < \hole, \\
1, & \text{if } p > \hole,
\end{cases}
\end{equation*}
almost surely. 
\end{rem}
Theorem \ref{unique} also implies the analogue of the estimation (\ref{supercritical}) in the supercritical phase. Note that the FKG inequality (Theorem \ref{FKG}) can be applied to the probability space $(\Omega, \F, P_p)$ of face percolation and the increasing event $\{|G_{x^*}| = \infty \}$.
\begin{theo}\label{theo:connection_in_supercritical}
\rm
If $p > \hole$, there exists $c := c(p) > 0$ such that
\begin{align*}
P_p(x^* \overset{\text{hole}} \longleftrightarrow y^*) \geq c  \ \text{for any } x^*, y^*  \in (\Z^d)^*.
\end{align*} 
\end{theo}
\begin{proof}
If $p > \hole$, then there exists the unique infinite hole cluster almost surely and we obtain
\begin{align*}
P_p(x^* \connect y^*) \geq P_p(|G_{x^*}| = \infty,\,  |G_{y^*}| = \infty)
\end{align*}
for any $x^*, y^*  \in (\Z^d)^*$. We may apply the FKG inequality to the increasing events $\{ |G_{x^*}| = \infty \}$, $\{|G_{y^*}| = \infty \}$ to find that the right hand side is bounded below by
\begin{align*}
P_p(|G_{x^*}| = \infty)P_p(|G_{y^*}| = \infty) = \phole(p)^2 > 0,
\end{align*}
which does not depend on $x^*, y^*  \in (\Z^d)^*$. 
\end{proof}
\begin{rem}
\rm
In the subcritical $p < \hole$, it is easy to show that the probability $P_p(x^* \connect y^*)$ converges to 0 as $\|x^* - y^*\|_1 \longrightarrow \infty$. Indeed, it is clear that 
\begin{align*}
P_p(0^* \connect x^*) \leq P_p(0^* \connect \B(n))
\end{align*}
for $2n \leq \|0^* - x^*\|_1$. Since $P_p( \bigcap_n \{ 0^* \connect \B(n) \}) = \phole(p) = 0$, the right hand side converges to 0 as $n \longrightarrow \infty$. 
\end{rem}
\section{Estimates of the critical probability}
\subsection{Bond percolation on the dual lattice}
\label{sec:hyouka}
In this subsection, we explain the main idea for the proof of Theorem \ref{upper}. Let $e^* \in (\mathbb{E}^d)^*$ be a dual bond given by
\begin{align*}
e^* = \big < x^* , x^* + (0,0, \ldots ,\overset{i}{\check{1}}, \ldots , 0)  \big>
\end{align*}
for some $x^* = ((x^*)_1, \ldots, (x^*)_d ) \in (\Z^d)^*$, and let the face $Q_{e^*} \in \K^d_{d-1}$ be defined by
\begin{align*}
Q_{e^*} = &[(x^*)_1-1/2, (x^*)_1+1/2] \times [(x^*)_2-1/2, (x^*)_2+1/2] \times \\
                &\cdots \times \overset{i}{\check{[(x^*)_i + 1/2]}} \times \cdots \times [(x^*)_d-1/2, (x^*)_d+1/2].
\end{align*}
Note that $Q_{e^*}$ is a unique face intersecting $e^*$ (see Figure \ref{bijection}). 
\begin{figure}[H]
  \centering
  \includegraphics[width=4cm]{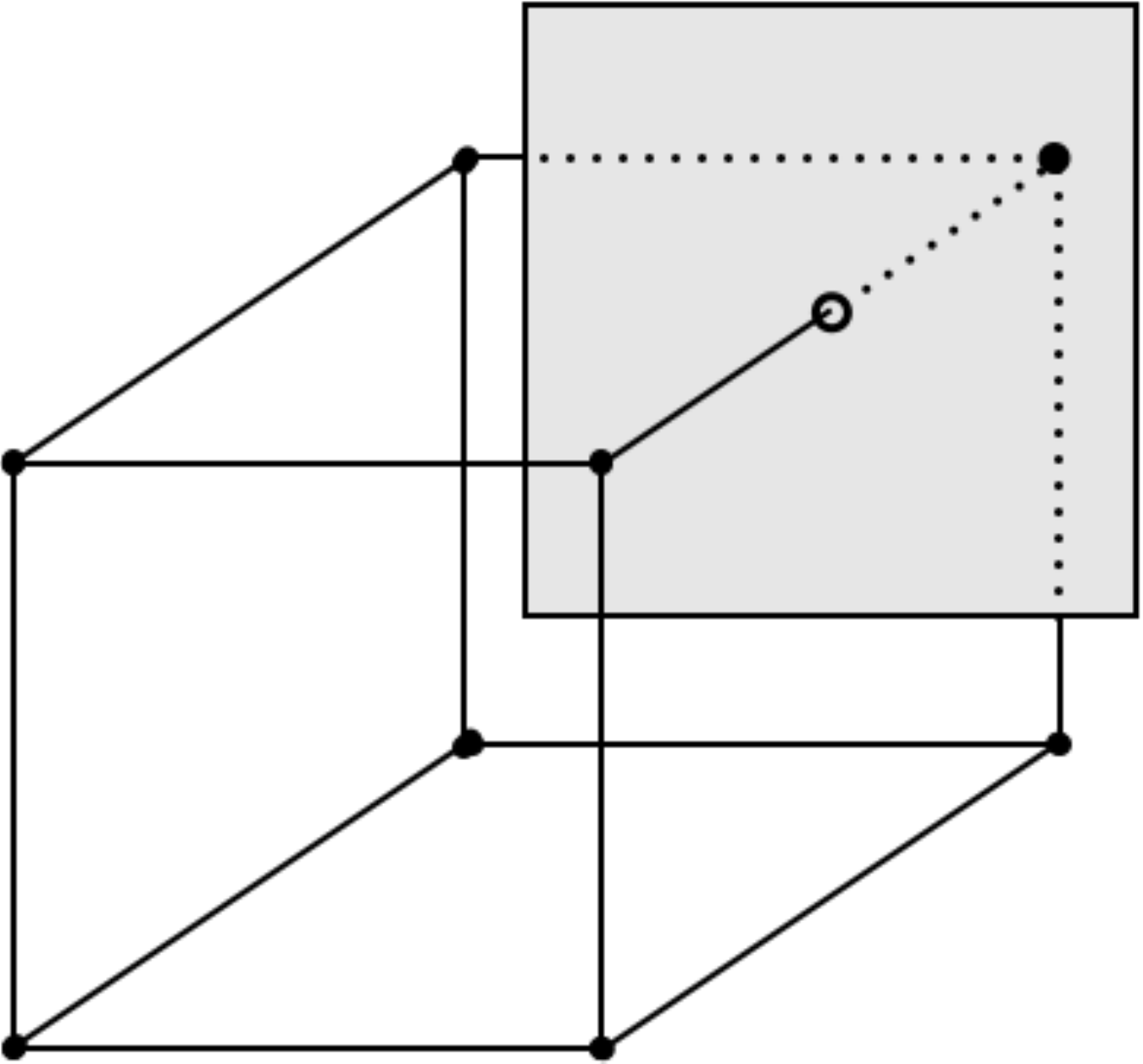}
  \caption{$d=3$. Dual bonds (black) and the intersecting face (gray)}
  \label{bijection}
\end{figure}
We define the configuration of dual bonds  in $(\mathbb{L}^d)^*$ as 
\begin{align}
e^* \in (\mathbb{L}^d)^*: \operatorname{open} \Longleftrightarrow Q_{e^*} \in \mathcal K_{d-1}^d: \operatorname{closed}.
\label{taiou}
\end{align}
This induces bond percolation on $(\mathbb{L}^d)^*$ with probability $1-p$.
From now on, when we consider face percolation, dual bond percolation is also considered by (\ref{taiou}).
\par
When $d = 2$, as we may expect from Figure \ref{hole_bijection}, the hole constructed by open faces corresponds to the finite cluster on $(\mathbb{L}^d)^*$.
\begin{figure}[H]
  \centering
  \includegraphics[width=6cm]{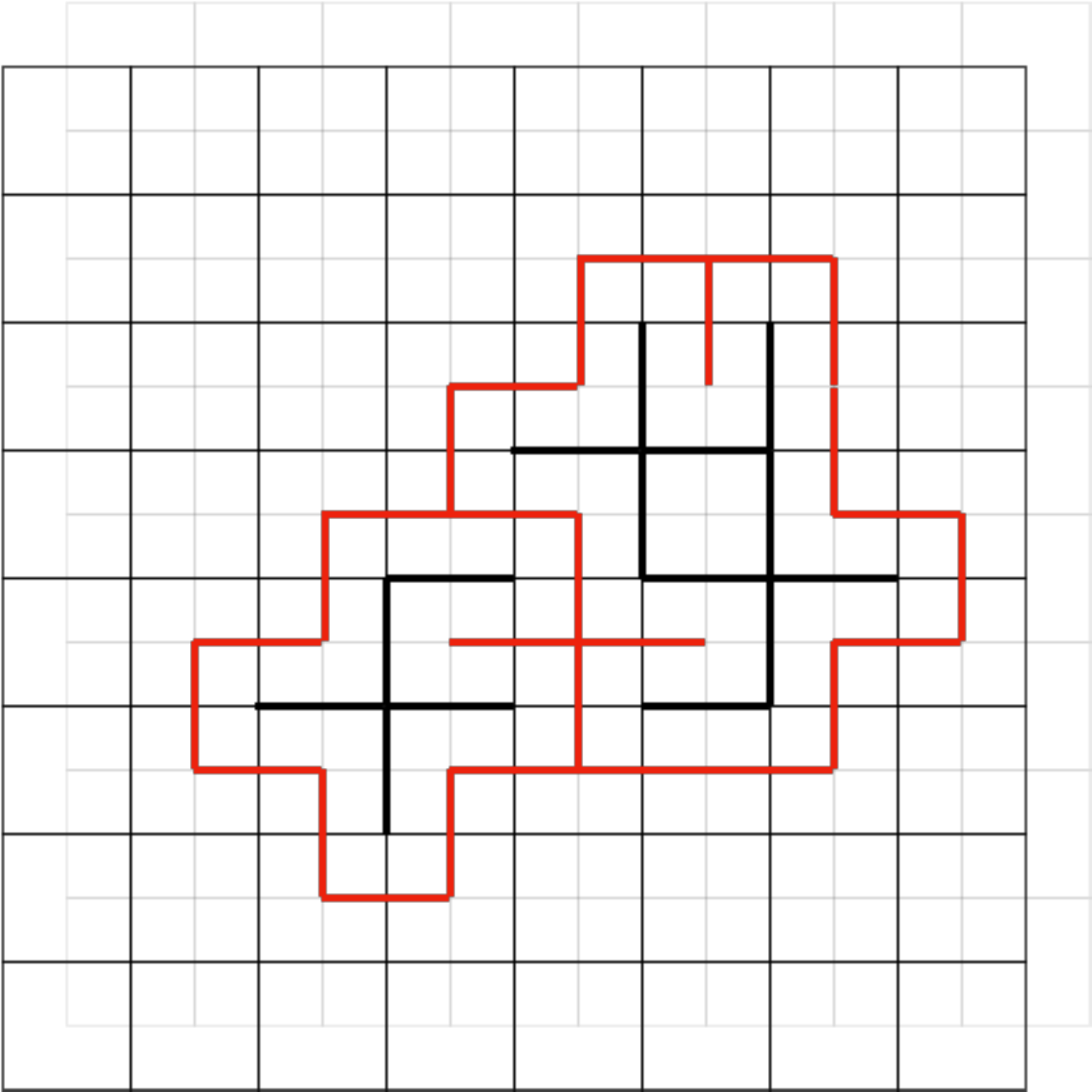}
  \caption{Finite clusters in the dual lattice (black) and faces constructing holes (red)}
  \label{hole_bijection}
\end{figure}
We prove this correspondence in general dimension $d \geq 2$.
\par
\begin{prop}
\rm
\label{prop:hole_bijection}
The holes bijectively correspond to the finite dual cluster in $(\mathbb{L}^d)^*$. Moreover, under this bijective correspondence, two holes $D$, $D^{\prime}$ are adjacent if and only if the corresponding finite clusters $C^*$, $(C^*)^{\prime}$ share some boundary edges, i.e., $\Delta C^* \cap \Delta (C^*)^{\prime} \neq \emptyset$.
\begin{table}[H]
 \centering
  \begin{tabular}{c|c}
    face & dual lattice\\ \hline \hline
    hole & finite cluster  \\ 
   adjacency of holes & sharing some boundary edges of two finite clusters  \\ 
  \end{tabular}
\end{table}
\end{prop}
\par
It follows from Proposition \ref{prop:hole_bijection} that the problem of hole generation can be studied as the problem of the bond percolation process.  \par
For a dual vertex $x^* \in (\Z^d)^*$, we denote its {\it occupied cell} $B_{x^*} \subset \R^d$ by 
\begin{align*}
B_{x^*} := \Pi_{i=1}^{d} [(x^*)_i -1/2, (x^*)_i +1/2].
\end{align*}
Then for a dual bond $e^* = \big< x^* , y^* \big>$, it is easy to see
\begin{align}
\label{face=box}
Q_{e^*} = B_{x^*} \cap B_{y^*},
\end{align}
and the interiors of two different occupied cells do not intersect. \par
For a dual cluster $C^*$, we denote the union of occupied cells by $B_{C^*} = \bigcup_{x^* \in C^*} B_{x^*}$. \par
For the proof of Proposition \ref{prop:hole_bijection}, we show the following lemma.
\begin{lemm}
\rm
\label{bounded}
Let $C^*$ be a dual cluster in $(\mathbb{L}^d)^*$. Then, 
\begin{align}
\partial B_{C^*} = \bigcup_{e^* \in \Delta C^*} Q_{e^*}.
\label{eq:boundary}
\end{align}
\end{lemm}
\begin{proof}
Let $a \in \R^d$ be an arbitrary point in $Q_{e^*}$ for some $e^* \in \Delta C^*$. We can set $e^* = \big< x^*, y^* \big> $ for some $x^* \in C^*$ and $y^* \notin C^*$. From (\ref{face=box}), $a$ must be in $B_{x^*} \cap B_{y^*}$, in particular $a \in B_{C^*}$. For any $\epsilon > 0 $, the $\epsilon \mathchar `-$open ball $B(a;\epsilon)$ centered at $a$  intersects the interior of $B_{y^*}$. Since the interior of $B_{y^*}$ is disjoint with $B_{C^*}$, we see $a \in \partial B_{C^*}$. \par
Conversely, suppose $a \in \partial B_{C^*}$. Then there exists $x^* \in C^*$ and $y^* \notin C^*$ such that $a \in B_{x^*} \cap B_{y^*}$. We write 
\begin{align*}
&a = (a_1, a_2, \ldots ,a_d), \\
&x^* = ((x^*)_1, (x^*)_2, \ldots ,(x^*)_d), \\
&y^* = ((y^*)_1, (y^*)_2, \ldots ,(y^*)_d).
\end{align*}
Since $a$ is on some faces, at least one element of $a_1, a_2, \ldots ,a_d$ is an integer. Suppose that $a_{i_1},  \ldots, a_{i_k}$($1 \leq i_1 < \cdots < i_k \leq d$) are integers. Then we can write
\begin{align*}
a_{i_j} = (x^*)_{i_j} \pm 1/2 = (y^*)_{i_j} \pm 1/2
\end{align*}
for $j = 1, \ldots , k$, and
\begin{align*}
a_i \in ((x^*)_i -1/2, (x^*)_i +1/2) \cap ((y^*)_i -1/2, (y^*)_i +1/2) 
\end{align*}
for the other coordinates. Note that any elements of a dual vertex is expressed as $n \pm 1/2$ for some integer $n$. Thus we obtain 
\begin{align*}
(x^*)_{i_j}  = (y^*)_{i_j} \, \text{or} \, (x^*)_{i_j}  = (y^*)_{i_j} \pm 1
\end{align*}
for $j = 1, \ldots , k$, and
\begin{align*}
(x^*)_i = (y^*)_i
\end{align*}
for the other coordinates. Let us retake the coordinates $l_1, \ldots , l_m$ for $(x^*)_{l_i} \neq (y^*)_{l_i}$. Then, we may write
\begin{align*}
(y^*)_{l_j} = (x^*)_{l_j}  + s_{l_j}, 
\end{align*}
where $s_{l_j}$ is equal to $\pm 1$. Let the sequence $x^* = (z^*)^{1}, (z^*)^{2}, \ldots , (z^*)^{m} = y^*$ of dual vertices be as follows.
\begin{align*}
(z^*)^{j+1} = (z^*)^{j}  + (0,0, \ldots ,\overset{l_j}{\check{s_{l_j}}}, \ldots , 0).
\end{align*}
Then, $a \in B_{(z^*)^{j}}$ for any $(z^*)^{j}$. Let $(z^*)^{j}$ be the dual vertex satisfying $(z^*)^j \in C^*$ and $(z^*)^{j+1} \notin C^*$. Then $e^* := \big<(z^*)^{j}, (z^*)^{(j+1)} \big>$ is the edge of $\Delta C^*$ and from (\ref{face=box}), we see that $a \in Q_{e^*}$.
\end{proof}
We give the proof of Proposition \ref{prop:hole_bijection}.
\begin{proof}[Proof of Proposition \ref{prop:hole_bijection}] 
Let $C^*$ be a finite cluster in $(\mathbb{L}^d)^*$. First, by using Lemma \ref{bounded}, we show that any vertices in $C^*$ belong to the same hole. Note that for each boundary edge $e^* \in \Delta C^*$, the corresponding face $Q_{e^*}$ is open.  \par
Take an arbitrary dual vertex $x^* \in C^*$. Let $\gamma$ be a path from $x^*$ to infinity, i.e., a continuous map $\gamma: [0,\infty) \longrightarrow \R^d$ with $\gamma(0) = x^*$ whose image is unbounded. Then, we have
\begin{align*}
t^{\prime} := \inf \{ t : \gamma(t) \notin B_{C^*}\} < \infty
\end{align*}
since $B_{C^*}$ is bounded. \par 
Let us check that $\gamma(t^{\prime}) \in \partial B_{C^*}$. The continuity of $\gamma$ implies $\gamma(t^{\prime}-1/n) \longrightarrow \gamma(t^{\prime})$ as $n \longrightarrow \infty$. Since $\gamma(t^{\prime}-1/n) \in B_{C^*}$ and $B_{C^*}$ is closed, $\gamma(t^{\prime})$ is in $B_{C^*}$. Moreover, for any $\epsilon > 0$, there exists $\delta > 0$ such that 
\begin{align*}
|t-t^{\prime}| < \delta \Longrightarrow \| \gamma(t) - \gamma(t^{\prime})\|_2 < \epsilon.
\end{align*}
From the definition of $t^{\prime}$, there exists $t$ such that $t^{\prime} \leq t < t^{\prime} + \delta$ and $\gamma(t) \notin B_{C^*}$. For this $t$, we see $\| \gamma(t) - \gamma(t^{\prime})\|_2 < \epsilon$ and then, $B(\gamma(t^{\prime}) ; \epsilon) \cap (\R^d \setminus B_{C^*}) \neq \emptyset$. Thus we have $\gamma(t^{\prime}) \in \partial B_{C^*}$. \par 
From Lemma \ref{bounded}, we obtain 
\begin{align*}
\gamma(t^{\prime}) \in \partial B_{C^*} = \bigcup_{e^* \in \Delta C^*} Q_{e^*}.
\end{align*}
This implies that the path $\gamma$ must hit some face in $K(\omega)$, since $Q_{e^*}$ is open for $e^* \in \Delta C^*$. Thus we see that $x^*$ is in a bounded connected component of $\R^d \setminus K(\omega)$, i.e., a hole. From the connectedness of $C^*$ and the fact that open dual bonds and open faces must be disjoint, vertices in $C^*$ must belong to the same hole, say $D_{C^*}$. \par
Next, let us check that $C^* \mapsto D_{C^*}$ is bijective. We construct the invertible map $D \mapsto C^*$.
For a hole $D$, it must have a dual vertex $x^*$. For this $x^*$, the  dual cluster $C^*(x^*) \subset (\mathbb{L}^d)^*$ including $x^*$ is finite, since the existence of an infinite path from $x^*$ contradicts the boundedness of $D$. It is sufficient to check that $C^*(x^*) = C^*(y^*)$ for two vertices $x^*, y^* \in D$. Suppose on the contrary $C^*(x^*) \neq C^*(y^*)$. Since $x^*, y^* \in D$, there exists a path $\gamma: [0,1] \longrightarrow \R^d$ in $D$ such that $\gamma(0) = x^*$ and $\gamma(1) = y^*$. Similarly to the  above discussion, we can show that $\gamma$ must intersect $\partial B_{C^*(x^*)} = \bigcup_{e^* \in \Delta C^*(x^*)} Q_{e^*}$. This means that $\gamma$ must hit some open face, which contradict the condition that $\gamma$ is in $D$.  Clearly, we can see $D = D_{C^*(x^*)}$ for this $x^*$. Therefore, $C^* \mapsto D_{C^*}$ is bijective.\par
Under this bijective correspondence, suppose that two finite dual clusters $C^*_1$, $C^*_2$ share a dual boundary $e^*$. From the above discussion, we see that the corresponding holes $D_{C^*_1}$, $D_{C^*_2}$ are constructed by the faces which correspond to the boundary edges of $C_1$, $C_2$, respectively. Therefore, $\partial D_{C_1}$, $\partial D_{C_2}$ share the face $Q_{e^*}$. This means the adjacency of the holes. Similarly we can see that if two holes are adjacent, the corresponding finite clusters share some boundary edges. 
\end{proof}

We apply Proposition \ref{prop:hole_bijection} to the hole percolation model. \par
Note that by Theorem \ref{th:uniqueness}, the infinite cluster in $(\mathbb{L}^d)^*$ is uniquely determined, say $I \subset (\mathbb{L}^d)^*$ (if there is no infinite cluster, $I \subset (\mathbb{L}^d)^*$ is assumed to be $\emptyset$). 
Let us define the subgraph $(\mathbb{L}^d)^* - I \subset (\mathbb{L}^d)^*$ as follows:
\begin{itemize}
\item{Vertex set: }$\{x^* \in (\Z^d)^*: x^* \text{ does not belong to } I \}$
\item{Edge set: }$\{e^* \in \mathbb{E}^d: e^* \text{ belongs to neither } I \text { nor } \Delta I\}$
\end{itemize}
Namely, $(\mathbb{L}^d)^* - I$ is the graph defined by removing the graph $I$ and its incident edges from $(\mathbb{L}^d)^*$.  \par
Let us observe the shape of a hole graph by using the definition (\ref{regarding_subgraph}), regarding a hole as the subset of $\Z^d$.  
From Theorem \ref{th:uniqueness} and Proposition \ref{prop:hole_bijection}, a dual vertex $x^* \in (\mathbb{L}^d)^* - I$ belongs to a finite dual cluster almost surely, and thus $x^*$ is included in some hole. On the other hand, from Proposition \ref{prop:hole_bijection}, a dual vertex $x^*$ in some hole must in a finite dual cluster, and thus in $x^* \in (\mathbb{L}^d)^* - I$. Therefore, it follows that $x^* \in (\Z^d)^*$ belongs to the hole graph if and only if $x^* \in (\mathbb{L}^d)^* - I$. Roughly speaking, if we ignore the shape of each hole and their adjacency, the ``external appearance" of the hole graph coincides with $(\mathbb{L}^d)^* - I$. \par
In more detail, the following lemma shows the relation between clusters of a hole graph and connected components of $(\mathbb{L}^d)^* - I$. 
\begin{lemm}
\rm
\label{obs:taiou}
Hole clusters bijectively correspond to the connected component of $(\mathbb{L}^d)^* - I$. Moreover, under this correspondence, a hole cluster is infinite if and only if the corresponding connected component is infinite.
\begin{table}[H]
 \centering
  \hspace{-2em}
  \begin{tabular}{c|c}
    Hole graph & $(\mathbb{L}^d)^* - I$\\ \hline \hline
   cluster & connected component  \\ \hline
   infinite cluster & infinite connected component  \\ 
  \end{tabular}
\end{table}
\end{lemm}
Note that the right hand side of the above table only refers to $(\mathbb{L}^d)^* - I \subset (\mathbb{L}^d)^*$, and does not consider whether the remaining bonds in $(\mathbb{L}^d)^* - I$ is open or not.  Figure 6 shows the relation between $(\mathbb{L}^d)^* - I$ and its induced hole graph. 
 \begin{figure}[H]
  \centering
   \includegraphics[width=13cm]{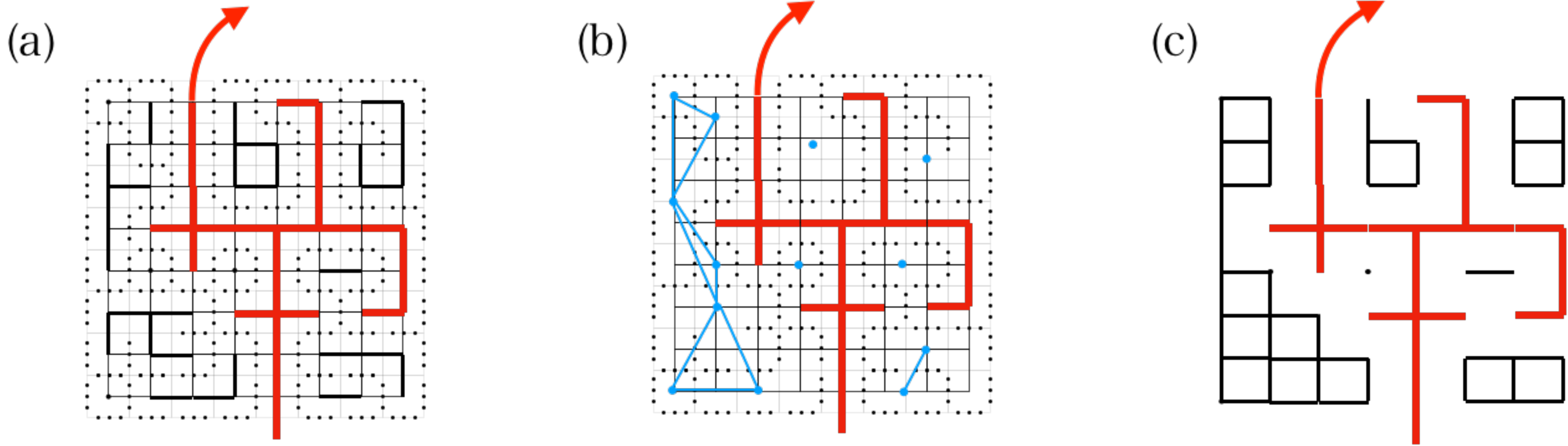}
  \caption{(a) the configuration of dual bonds, (b) the induced hole graph (blue), and (c) $(\mathbb{L}^d)^* - I$. Here $I$ is expressed by the red graph}
 \end{figure}

\begin{proof}
We first prove the following claim.
\begin{itemize}
\item[(A)] Two dual vertices $x^*$, $y^*$ in $(\mathbb{L}^d)^* - I$ belong to a same connected component if and only if the holes $D$, $D^{\prime}$ with $x^* \in D$, $y^* \in D^{\prime}$ belong to a same cluster of the induced hole graph. 
\end{itemize}
For two dual vertices $x^*$, $y^*$ which belong to a same connected component in $(\mathbb{L}^d)^* - I$, there exist holes $D_{x^*}$ and $D_{y^*}$ which include $x^*$ and $y^*$, respectively, because of the uniqueness of the infinite cluster (Theorem \ref{th:uniqueness}). Take a path from $x^*$ to $y^*$ in $(\mathbb{L}^d)^* - I$
\begin{align*}
x^*= x_0^*,e_1^*, x_1^*,e_2^* \ldots ,e_n^*,x_n^*=y^*,
\end{align*}
where each $x_i^*$ is a vertex of $(\mathbb{L}^d)^* - I$. Let $D_{x_i^*}$ be the hole  including  $x_i^*$. Then for $i = 0, 1, \ldots n-1$, we can see $D_{x_i^*} = D_{x_{i+1}^*}$ or $D_{x_i^*} \backsim D_{x_{i+1}^*}$. Indeed, if $D_{x_i^*} \neq D_{x_{i+1}^*}$, then the face $Q_{e^*_{i+1}}$ is open, and it is a common boundary of $D_{x_i^*}$ and $D_{x_{i+1}^*}$.  Choose the different holes along the path, we obtain
\begin{align*}
D_{x^*} \backsim D_{i_1} \backsim  \ldots \backsim D_{i_k} \backsim D_{y^*}.
\end{align*}
This means that $D_{x^*}$ and $D_{y^*}$ belong to the same cluster. \par
Let us show the sufficient condition. Without loss of generality, we can simply assume $D \backsim D^{\prime}$. For two dual vertices $x^*$, $y^*$ with $x^* \in D$, $y^* \in D^{\prime}$, respectively, we can take a path
\begin{align*}
x^* ,e_1^*, \ldots ,e_i^*,e_Q^*,f_1^*, \ldots ,f_j^*,y^*, 
\end{align*}
where $e_1^*, \ldots ,e_i^*$ and $f_1^*, \ldots ,f_j^*$ is the dual bonds in $D$, $D^{\prime}$, respectively, and $e_Q^*$ is the dual bond which corresponds to a common boundary face $Q$ of $D$ and $D^{\prime}$. Clearly, $e_1^*, \ldots ,e_i^*$ and $f_1^*, \ldots ,f_j^*$ is the bonds in $(\mathbb{L}^d)^* - I$. Since both end vertices of $e_Q$ are in $D$, $D^{\prime}$, $e_Q$ also belongs to $(\mathbb{L}^d)^* - I$. Thus $x^*$ and $y^*$ belong to the same connected component. \par
From the claim (A), we can construct the bijection between hole clusters and connected components in $(\mathbb{L}^d)^* - I$.  \par
For a connected component in $(\mathbb{L}^d)^* - I$, let us take its vertex $x^*$. From Theorem \ref{th:uniqueness}, there exists a hole $D_{x^*}$ including $x^*$. We set the corresponding hole cluster as the one including $D_{x^*}$. The claim (A) ensures well-definedness of this correspondence. Conversely, given a hole cluster, take its hole $D$ and a vertex $x^*$ in $D$. We set the corresponding connected component in $(\mathbb{L}^d)^* - I$ as the one including $x^*$. Again the claim (A) also ensures well-definedness, and hence we see that these are inverses of each other. 
\end{proof}

\subsection{Estimates of the critical probability}
In this subsection, we give the proof of Theorem \ref{upper}. First, we introduce the notations which will be used later. Let $X \subset \R^d$ be a cubical set constructed of only faces. We say that $X$ {\it encloses} the subset $V \subset (\Z^d)^*$ if and only if we may choose faces $Q_1,Q_2, \ldots ,Q_k$ of $X$ such that $V$ is included in a bounded domain of $\R^d \setminus (\bigcup_{i=1}^{k} Q_i)$. \par
Intuitively, it may seem to be true that $\face \leq \hole$. Indeed, if there exists an infinite path of holes, we may expect that the faces of the holes also make an infinite path. However, there is a counterexample shown in Figure 5. The key point of the proof for $\face \leq \hole$ is to check that the case like Figure \ref{fig:daunut} cannot influence the value of $\hole$. For this, we give the following lemma, which states that holes do not tend to be large.
\begin{figure}[H] 
   \centering
   \includegraphics[width=6cm]{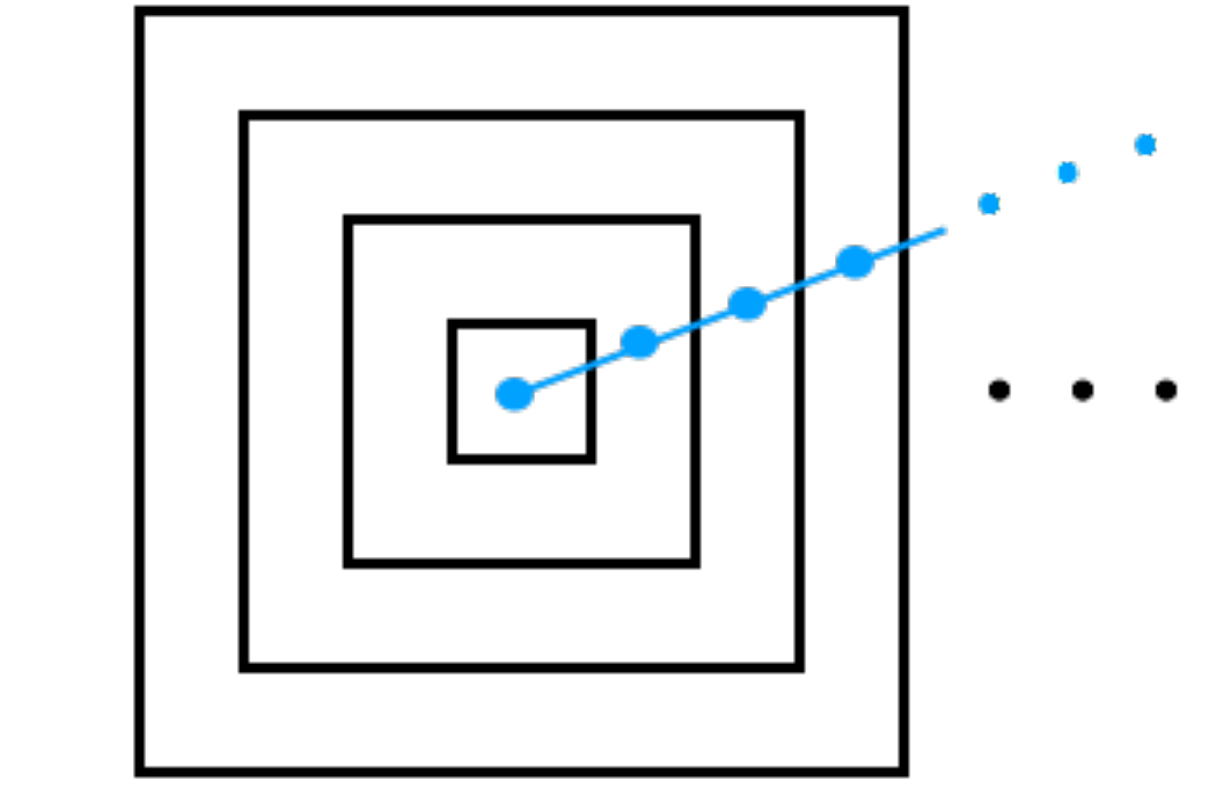} 
   \caption{A cubical set (black) and the induced hole graph (blue). The hole graph is infinitely connected while all face clusters are finite}
   \label{fig:daunut}
\end{figure}

\begin{lemm}
\rm
If $p < 1- p_c^{\operatorname{bond}}(d)$, then
\begin{align*}
P_p( K(\omega) \ \text{encloses} \ \B(n) \}) \longrightarrow 0
\end{align*}
as $n \longrightarrow 0$. 
\label{bighole}
\end{lemm}
\begin{proof}
Suppose $p < 1- p_c^{\operatorname{bond}}(d)$. Since $1 - p > p_c^{\operatorname{bond}}(d)$, there exists an infinite dual cluster almost surely. Fix a dual vertex $x^* \in (\mathbb{L}^d)^*$ of the infinite dual cluster. For sufficiently large $N \in \mathbb{N}$, we find $x^* \in \B(N)$ and  $K(\omega)$ does not enclose $\B(N)$ (see Figure \ref{fig:surround}).
\begin{figure}[H]
  \centering
  \includegraphics[width=6cm]{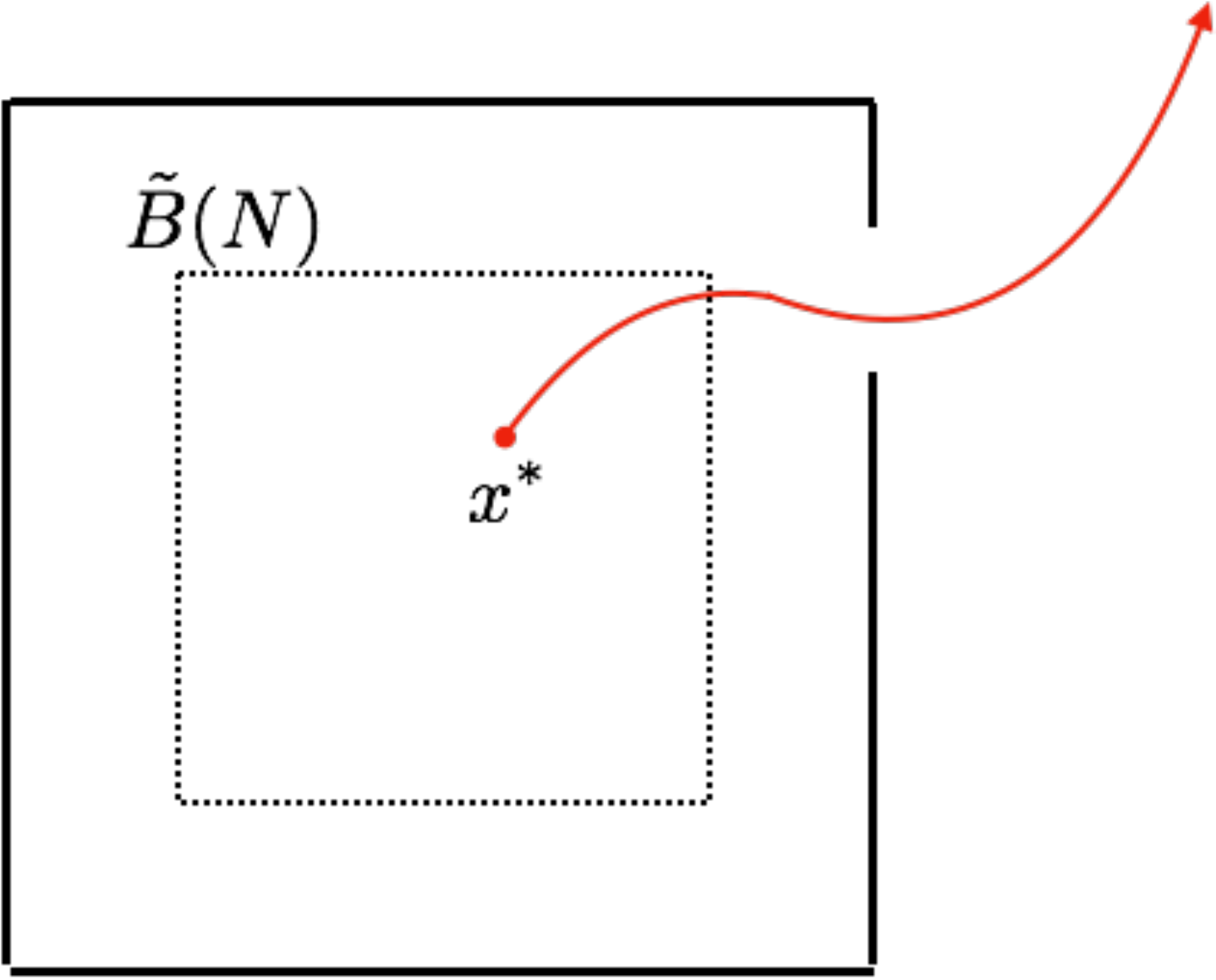}
  \caption{The infinite cluster in the dual lattice (red), $\B(N)$ (dotted), and faces in $K(\omega)$ (thick)}
   \label{fig:surround}
\end{figure}
Thus, we obtain
\begin{align*}
P_p(\bigcap_{n \in \mathbb{N} } \{ K(\omega) \ \text{encloses} \ \B(n) \}) = 0,
\end{align*}
which implies the statement of the lemma.
\end{proof}
Next, we prove Theorem \ref{upper}. 
\begin{proof}[Proof of Theorem \ref{upper}]
The upper bound follows from Lemma \ref{obs:taiou}. Indeed, suppose $1- \bond < p$. Then $1-p < \bond$ and that implies there exist no infinite dual clusters almost surely. This means that the hole graph consists of one infinite cluster spreading in $\R^d$．\par
We show the lower bound by using Lemma \ref{bighole}. Suppose that a sample $\omega \in \Omega$ satisfies $|G_{0^*}(\omega)|=\infty$. Then the following occurs: 
\begin{align*}
&\mbox{For any $M \in \mathbb{N}$, there exists a face cluster $C$ such that } \\
&(1)\, |C| \geq M \ \ (2)\, \mbox{$C$ encloses $0^*$}
\end{align*}
Let us denote this event by $A$. Then it suffices to show that $P_p(A)=0$ for $p < \face$. For $n \geq 0$ and $m \geq 1$, we denote by $A_{m,n}$ the event that
\begin{align*}
 A_{m,n}:= \{ &\mbox{there exists a face cluster $C$ such that} \\ 
                     &(1)\, |C| \geq m \ (2)\, \mbox{$C$ encloses $\B(n)$ and $0^*$} \}.
\end{align*}
Then $A_{m,n}$ is nonincreasing with respect to $n$, $m$, respectively. Since $\bigcap_m A_{m,0} = A$, it suffices to show that
\begin{align}
\lim_{m \to \infty}P_p(A_{m,0}) = 0.
\label{lim_lower}
\end{align} 
\par
Fix arbitrary $\epsilon > 0$ and suppose $p < \face(d)$. Then, we have
\begin{align*}
1-p > 1 - \face(d) \geq 1/2 \geq \bond(d),
\end{align*}
where the second inequality follows from Proposition \ref{face_upper} and $\face(2) = \frac{1}{2}$, and the last inequality follows from Remark \ref{bond_dimension}. From Lemma \ref{bighole}, we can take an integer $N$ independently from $m$ such that
\begin{align*}
P_p(A_{m,N}) \leq P_p  \big( \{K(\omega) \mbox{ encloses $\B(N)$}\} \big)  < \epsilon.
\end{align*}
For this $N$, if the event $A_{m,0} \setminus A_{m,N}$ occurs, then there exists a face cluster $C$ intersecting $\Lambda^N = [-N, N]^d$ such that $|C| \geq m$. Therefore we obtain 
\begin{align*}
P_p(A_{m,0} \setminus A_{m, N}) &\leq \sum_{Q \subset \Lambda^N: \text{face} } P_p(|C(Q)| \geq m) \\
					    &= |\Lambda^N |  P_p(|C(Q)| \geq m).
\end{align*}
Here, we denote by $|\Lambda^N |$ the number of faces in $\Lambda^N$. Since $p < \face$, the last expression converges to $0$ as $m \longrightarrow 0$. Thus, for large enough $m$, we obtain
\begin{align*}
P_p(A_{m,0}\setminus A_{m,N}) < \epsilon
\end{align*}
and thus, we obtain
\begin{align*}
P_p(A_{m,0}) \leq P_p(A_{m,N}) + P_p(A_{m,0} \setminus A_{m,N}) < 2 \epsilon, 
\end{align*}
which completes the proof of (\ref{lim_lower}).
\end{proof}

\begin{rem}
\rm
\label{behave_of_theta2}
Suppose $d=2$. Since $\hole = 1/2$ (see Remark \ref{exact_1/2}), $\phole(p) = 0$ for $p < 1/2$.
Moreover, from the proof of Theorem \ref{upper}, we may have $\phole(p) = 1$ whenever $\pbond(1-p) = 0$.  From Remark \ref{rem:Kesten}, we see $\pbond(p) = 0$ for $p \leq 1/2$. Thus, $\phole(p)$ is determined as follows:
\begin{equation*}
\phole(p) = \begin{cases}
0, &  \text{if } p < 1/2, \\
1,  & \text{if } p \geq 1/2.
\end{cases}
\end{equation*}
\end{rem}
Furthermore, from the proof of Theorem \ref{upper}, we also see that for $p \geq 1/2$, the hole graph consists of one infinite cluster, spreading in $\R^d$

\section{Uniqueness of the infinite hole cluster}
We give the proof of Theorem \ref{unique} in this section. Note that from Remark \ref{ergodic}, the number of infinite clusters is the function of $p$. Clearly, when $p = 1$, the hole graph consists of one infinite cluster, and $\phole(p) = 0$ when $p = 0$. Let us suppose $0<p<1$ in this section. \par
Let us first study the case $d=2$. From Remark \ref{behave_of_theta2}, the uniqueness holds for $p \geq 1/2$. For $p < 1/2$, we have already seen $\phole(p) = 0$. Thus, Theorem \ref{unique} holds for $d= 2$. \par
In this section, we prove Theorem \ref{unique} for $d \geq 3$.
\par
Let the random variable $M_n$ be the number of infinite hole clusters intersecting $\B(n)$. Let the random variable $N_n(1)$ (resp.$N_n(0)$) be the number of infinite hole clusters when the faces in $\Lambda^n$ are set to be open (resp. closed). First, we show that the number $N_{\infty}$ of infinite hole clusters  should be $0,1$ or $\infty$.
\begin{lemm} 
\rm
\label{finite_case}
If $1 \leq k < \infty$, then $P_p(N_{\infty} = k) = 1$ implies $k=1$. 
\end{lemm}
\begin{proof}
Clearly, when $p = 1$, the hole graph consists of one infinite cluster and we find $k = 1$. When $p = 0$, the hole graph is $\emptyset$.  We suppose $0<p<1$ in this proof.
Suppose $P_p(N_{\infty} = k) = 1$. We have
\begin{align*}
1 = P_p(N_{\infty} = k) = \sum_{A_n}P_p(N_{\infty} = k \ | \  A_n) P_p(A_n), 
\end{align*}
where $A_n$ is a cylinder set determined by the configuration of all faces in $\Lambda^n$. The right hand side is the sum over all the configurations of faces in $\Lambda^n$. Since the number $|\Lambda^n|$ of faces in $\Lambda^n$ is finite and $0 < p < 1$, $P_p(A_n)$ is strictly positive．Thus for any $A_n$, $P_p(N_{\infty} = k \ | \ A_n)$ must be equal to $1$. In particular, when $A_n$ is the case that ``all faces in $\Lambda^n$ are open", we obtain
\begin{align*}
1 &= P_p(N_{\infty} = k \ | \  A_n) \\
   &= P_p(N_{\infty} = k \mbox{ and } \,A_n ) / P_p(A_n) \\
   &= P_p(N_n(1) = k \mbox{ and } \,A_n ) / P_p(A_n) \\
   &= P_p(N_n(1) = k),
\end{align*}
where we used the independence of $N_n(1)$ and $A_n$. Similarly, we have
\begin{align*}
   1= P_p(N_n(0) = k),
\end{align*}
and thus,
\begin{align*}
   1= P_p(N_n(1) = N_n(0) = k).
\end{align*}
We may see that $N_n(1) = N_n(0) < \infty$ implies $M_n \leq 1$, where we use the assumption $k < \infty$. Indeed, if on the contrary $M_n \geq 2$, the number of infinite hole clusters must decrease by opening all the faces in $\Lambda^n$, which contradicts $N_n(1) = N_n(0) < \infty$. From the relation
\begin{align*}
\{ N_{\infty} \leq 1\}  = \{\forall n: M_n \leq 1\} = \bigcap_n \{ M_n \leq 1 \}
\end{align*}
and continuity of measures, we obtain
\begin{align*}
1= P_p(M_n \leq 1)\longrightarrow P_p(N_{\infty} \leq 1)
\end{align*}
as $n \longrightarrow \infty$. This implies $P_p(N_{\infty} \leq 1) = 1$ and completes the proof of Lemma \ref{finite_case}.
\end{proof}
From Lemma \ref{finite_case}, $P_p(N_{\infty} = k) = 0$ for $2 \leq k < \infty$. Thus it is sufficient for Theorem \ref{unique} to show that infinite hole cluster is not infinite. To this aim, we need two more lemmas. The first one is also used to show the uniqueness of infinite bond clusters (for the proof, see \cite[Lemma 8.5]{Grimmett}). \par
Let us give the notations for the lemma. For a set $Y$ with $|Y| \geq 3$, a {\it 3-partition} $\Pi = \{ \Pi_1,\Pi_2,\Pi_3\}$ of $Y$ is a partition of $Y$ into exactly three non-empty sets $\Pi_1, \Pi_2, \Pi_3$. We say that two 3-partitions $\Pi = \{ \Pi_1,\Pi_2,\Pi_3\}$ and $\Pi^{\prime}= \{ \Pi_1^{\prime},\Pi_2^{\prime},\Pi_3^{\prime} \}$ are {\it compatible} if there exists an ordering of their elements such that  $\Pi_1 \sqcup \Pi_2 \subset \Pi_3^{\prime}$.  The lemma is expressed as follows.
\begin{lemm}
\rm
\label{lemm:compatible}
Let $Y$ be a set with $|Y| \geq 3$, and ${\mathcal P}$ be a set of 3-partitions of $Y$. If any two 3-partitions in ${\mathcal P}$ are compatible, then $|{\mathcal P}| \leq |Y| - 2$. 
\end{lemm}

Before giving the second lemma, we again set the notations. Let us say that $x^* \in (\mathbb{L}^d)^*$ is a {\it trifurcation} if: 
\begin{enumerate}
\item there exists a hole which includes only $x^*$, say $D_{x^*}$,
\item$D_{x^*}$ belongs to an infinite hole cluster $I$, and
\item the graph $I-D_{x^*}$ obtained by deleting the vertex $D_{x^*}$ and its incident edges from $I$ consists of   exactly three infinite clusters. 
\end{enumerate}
We denote by $T_{x^*}$ the event that $x^* \in (\Z^d)^*$ is a trifurcation. The second lemma is as follows.

\begin{lemm}
\rm
\label{lem:trifurcation=0}
Assume $d \geq 2$, then $P_p(T_{0^*})  = 0$.
\end{lemm}

\begin{proof}
Let $K \subset G$ be a cluster of the hole graph $G$. Assume that $x^* \in K \cap \B(n)$ is a trifurcation. Then $K$ is infinite and the deleted graph $K - D_{x^*}$ consists of exactly three infinite clusters, say $K_1, K_2, K_3$. Then ${x^*}$ induces a 3-partition $\Pi(x^*) := \{ K_i \cap \partial \B(n+1) : i = 1,2,3 \}$ of $K \cap \partial \B(n+1)$. Moreover, for two trifurcation $x^* , y^* \in K \cap \B(n)$, we show that $\Pi (x^*)$ and $\Pi (y^*)$ are compatible (see Figure \ref{compatible}).  We set 
\begin{align*}
\Pi(x^*) := \{ K_i \cap \partial \B(n+1) : i = 1,2,3 \}, \\
\Pi(y^*) := \{ K_i^{\prime} \cap \partial \B(n+1) : i = 1,2,3 \}, 
\end{align*}
respectively. Without loss of generality, we may assume $K_1$, $ K_1^{\prime}$ includes $y^*, x^*$, respectively. It is sufficient to see that
\begin{align*}
[K_2^{\prime} \cap \partial \B(n+1)] \cup [K_3^{\prime} \cap \partial \B(n+1)] \subset K_1 \cap \partial \B(n+1), 
\end{align*}
which can be reduced to the following relation as the graphs
\begin{align*}
K_2^{\prime} \cup K_3^{\prime} \subset K_1.
\end{align*}
From the definition of a trifurcation, the graph $K_2^{\prime} \cup K_3^{\prime} \cup D_{y^*}$ is an infinite cluster, which does not include $D_{x^*}$. From the setting of $\Pi(x^*)$, it must be included in one of $K_i \ (i = 1,2,3)$.  Since it includes $D_{y^*}$, we can see $K_2^{\prime} \cup K_3^{\prime} \cup D_{y^*} \subset K_1$. 

\begin{figure}[H]
\centering
\includegraphics[width=11cm]{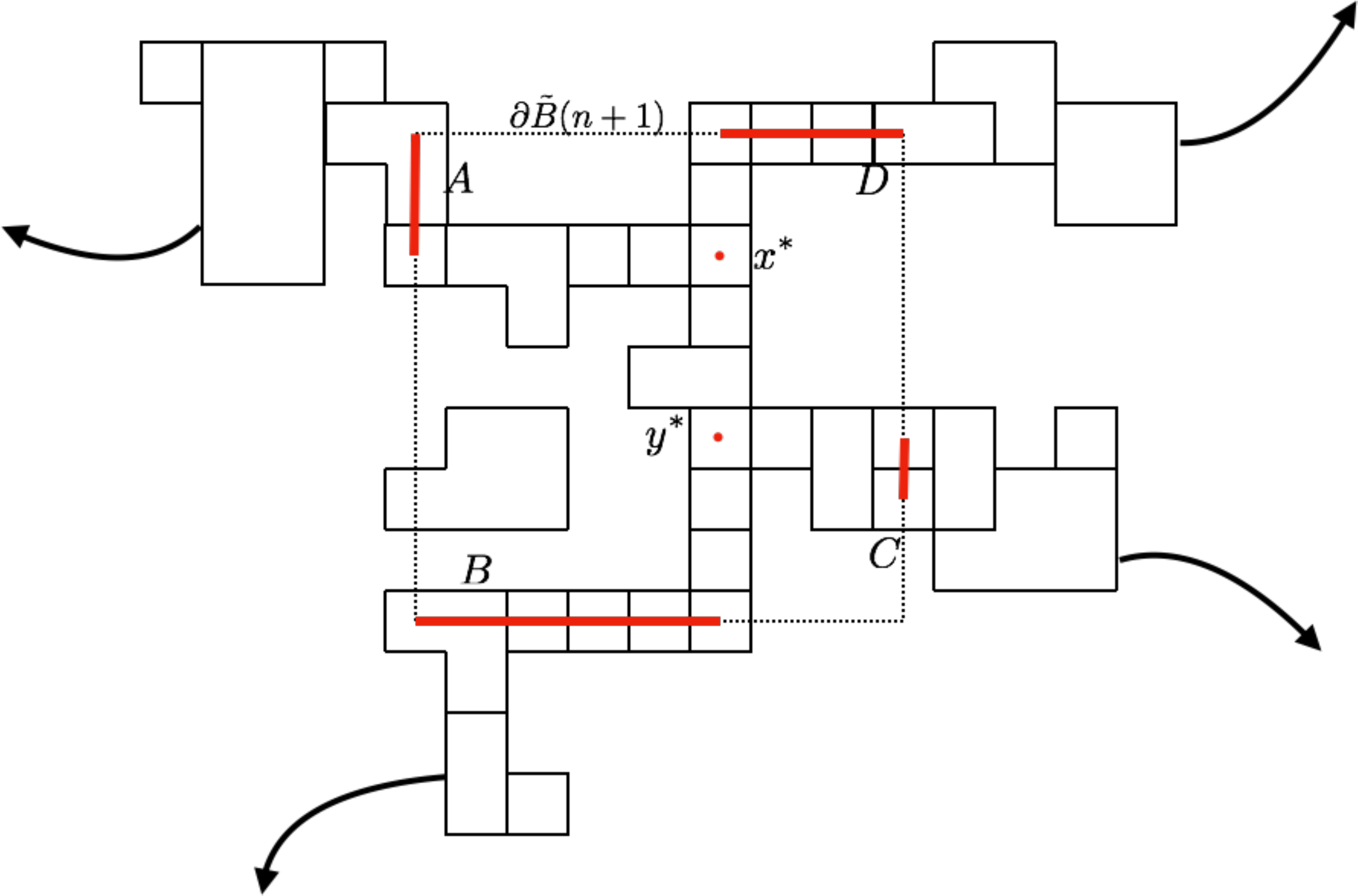}
\caption{$\Pi (x^*)=\{A,D,(B,C)) \}$ and $\Pi (y^*)=\{(A,D),B,C \}$}
\label{compatible}
\end{figure}
Here, we can see
\begin{align*}
    \# \{ x^* \in K \cap \B(n) : \mbox{trifurcation} \} = \# \{ \Pi(x^*) :  x^* \in K \cap \B(n) :  \mbox{trifurcation}\}.
\end{align*}
By using Lemma $\ref{lemm:compatible}$, the right hand side is bounded above by
 $|K \cap \partial \B(n+1)|$.
We take the sum over all  clusters $K \subset G$ intersecting $\B(n+1)$, to find that 
\begin{align*}
\# \{ x^* \in \B(n) : \mbox{trifurcation} \}  \leq  |\partial \B(n+1)|,
\end{align*}
which can be written as
\begin{align*}
\sum_{x^* \in  \B(n)} I_{T_{x^*}}   \leq  |\partial \B(n+1)|. 
\end{align*}
Then we take the expectation to find that
\begin{align*}
 |\B(n) |P_p(T_{0^*})   \leq  |\partial \B(n+1)|.
\end{align*}
By letting $n \longrightarrow \infty$, this gives us $P_p(T_{0^*})=0$. 
\end{proof}
By using Lemma \ref{lem:trifurcation=0}, we give the proof of Theorem \ref{unique}. 
\begin{proof}[Proof of Theorem \ref{unique}]
From Lemma \ref{lem:trifurcation=0}, it is sufficient to show that $P_p(N_{\infty} = \infty)= 1$ implies $P_p(T_{0^*})>0$.  \par
Since $P_p(M_n \geq 3) \longrightarrow P_p(N_{\infty} \geq 3)= 1$ as $n \longrightarrow \infty$, there exists $n$ such that 
\begin{align*}
P_p(M_n \geq 3) \geq 1/2.
\end{align*}
Suppose $M_n \geq 3$. We now show that we can make $0^*$ a trifurcation by changing the configuration of faces in $\Lambda^n$ properly (see Figure \ref{fig:irekae}). \par
First, we can take three dual vertices $a_i^* \ (i = 1, 2, 3) \in \B(n+1)$ satisfying the following conditions: 
\begin{itemize}
\item[(a)] $a_i^*$'s are included in distinct infinite hole clusters, and
\item[(b)] each $a_i^*$ is adjacent to $\B(n)$. 
\end{itemize}
From condition (a), we can see $a_i^*$ and $a_j^*$ are not adjacent for $i \neq j$ (if not, they must be in the same hole cluster). We also take three paths $\pi_i$ of dual lattices such that
\begin{itemize}
\item[(c)] each $\pi_i$ connects between $0^*$ and $a_i^*$ in $\B(n)$ ($i = 1, 2, 3$), and
\item[(d)] they repel each other, i.e., for any $i \neq j$, $x^* \in \pi_i \setminus 0^*$ and $y^* \in \pi_j \setminus 0^*$ are not adjacent. 
\end{itemize}
Denote by $I_i \ (i = 1, 2, 3)$ the infinite hole cluster including $a_i^*$. We change the configuration of faces in $\Lambda^n$ as follows: 
\begin{itemize}
\item[(i)] For each vertex of $\pi_i$, all nearest faces are open,
\item[(ii)] the faces in the boundary of $\Lambda^n$ are open whenever they are included in $I_i \ (i = 1, 2, 3)$, and
\item[(iii)] other faces are all closed.
\end{itemize}
Then $0^*$ becomes a trifurcation. Indeed, we can see that $0^*$ satisfies conditions 1 and 2 of the trifurcation. Let us check the third condition. Now, in $\Lambda^n$, $I_i$'s are connected only at $D_{0^*}$. On the outside of $\Lambda^n$, from the assumption (a) of $a_i^*$, there are no holes connecting different $I_i$'s. It remains to rule out the case that there appears a new hole constructed by faces both inside and outside of $\Lambda^n$ by the process (i), (ii) and (iii). From the assumption $d \geq 3$, the dual vertices in $\B(n) \setminus \pi_1 \cup \pi_2 \cup \pi_3$ are all connected, and the new faces do not contribute to make such holes.   \par
Finally, we show that
\begin{align*}
P_p(T_{0^*}) \geq P_p(M_n \geq 3) \operatorname{min} \{p,1-p\}^{| \Lambda^n|} > 0, 
\end{align*}
where we denote by $| \Lambda^n|$ the number of faces in $\Lambda^n$. This contradicts $P_p(T_{0^*}) = 0$, and completes the proof of Theorem \ref{unique}.
Let $\omega^n \in \Pi_{Q \subset \Lambda^n : \text{face}} \{0, 1\}$ be a configuration of faces in $\Lambda^n$. We denote by $T_{0^*}(\omega^n)$ the event that $0^*$ becomes a trifurcation when the configuration in $\Lambda^n$ are set to be $\omega^n$. For a configuration $\omega$, we also denote by $\left.\omega \right|_{\Lambda^n} \in \Pi_{Q \subset \Lambda^n : \text{face}} \{0, 1\}$ the restriction of $\omega$ to the faces in $\Lambda^n$.
Then we may write
\begin{align*}
P_p(T_{0^*}) = \sum_{\omega^n} P_p(T_{0^*}(\omega^n) \text{ and }  \left.\omega \right|_{\Lambda^n}= \omega^n), 
\end{align*}
where the right hand side is the sum of all configurations in $\Lambda^n$. Since the events $T_{0^*}(\omega^n)$ and $\{ \left.\omega \right|_{\Lambda^n}  = \omega^n\}$ are independent, the right hand side of this equation is bounded below by
\begin{align}
\label{ineq:T0_1}
\sum_{\omega^n } P_p(T_{0^*}(\omega^n)) P_p( \left.\omega \right|_{\Lambda^n}= \omega^n)
\geq \operatorname{min} \{p,1-p\}^{| \Lambda^n|} \sum_{\omega^n } P_p(T_{0^*}(\omega^n)).
\end{align}
From the above discussion, we have $\{M_n \geq 3\} \subset \bigcup_{\omega^n} T_{0^*}(\omega^n)$, and thus, the right hand side of (\ref{ineq:T0_1}) is again bounded below by
\begin{align*}
P_p(M_n \geq 3) \operatorname{min} \{p,1-p\}^{| \Lambda^n|} > 0.
\end{align*}

\begin{figure}[H]
\centering
\includegraphics[width = 15cm]{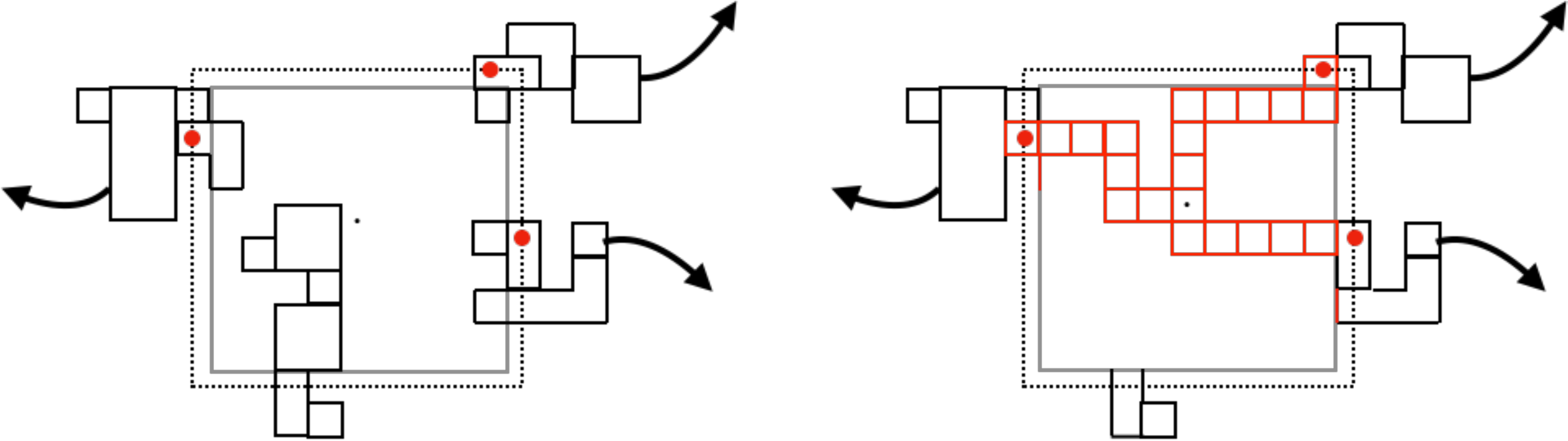}
\caption{The configuration of $M_n \geq 3$ (left) and $0^*$ is a trifurcation (right)}
\label{fig:irekae}
\end{figure}
\end{proof}

\begin{rem}
\rm
Grimmett, Holroyd and Kozma \cite{finite_cluster} study percolation of finite clusters in the bond percolation model; they focus on whether the random subset $X := \{x \in \Z^d \,:\, |C(x)| <  \infty \}$ of $\Z^d$ has an infinite connected component or not. By combining with uniqueness of the infinite connected component of $X$ \cite[Theorem 4.3]{finite_cluster}, Theorem \ref{unique} immediately follows from Lemma \ref{obs:taiou}, though we prove it directly in this paper.
\end{rem}

\section{Other properties of hole percolation}
\subsection{The right continuity of $\phole(p)$}
For bond percolation on $\mathbb{L}^d$, the percolation probability $\pbond(p)$ has some properties of continuity as follows. 
\begin{prop}
\rm
\label{right_conti_bond}
$\pbond(p)$ is right continuous on the interval $[0,1]$.
\end{prop}
\begin{prop}
\rm
\label{left_conti_bond}
$\pbond(p)$ is left continuous on the interval $(\bond,1]$.
\end{prop}
We prove the analogue of the continuity for the hole percolation model, by using the similar technique used in the case of bond percolation.
\begin{prop}
\rm
\label{right_conti}
$\phole(p)$ is right continuous on $[0,1] \setminus \{1-\bond\}$.
\end{prop}

\begin{rem}
\rm
\label{rem:right_conti}
Clearly $\phole(p)$ is right continuous at $p= 0$ since $\hole > 0$. In the interval $(1 - \bond,1]$, from the proof of Theorem \ref{upper}, we can see $\phole(p)=1$ and in particular $\phole(p)$ is right continuous. \par
For $p = 1- \bond$, we easily see
\begin{align*}
&\phole(p) \mbox{ is right continuous at } p=1-\bond \\ 
\Longleftrightarrow  \  & \phole(1-\bond) = 1.
\end{align*}
This is equivalent to
\begin{align}
\label{conjecture}
\pbond(\bond) = 0.
\end{align}
Indeed, similarly to the proof of  Proposition \ref{upper}, we can see that (\ref{conjecture}) implies $\phole(1-\bond) = 1$. If $\pbond(\bond)>0$, then we find
\begin{align*}
\phole(1-\bond) &\leq P_{\bond}(\mbox{there exists a hole including } 0^*) \\
                     & = 1- \pbond(\bond) < 1. 
\end{align*}
Yet $(\ref{conjecture})$ is proven only when $d=2$ and $d \geq 19$.
\end{rem}
In order to prove Theorem \ref{right_conti}, we use the following lemma.
\begin{lemm}
\rm
For each $n$, 
$P_p(0^* \connect \partial \B(n))$ is continuous at $p \in (0,1- \bond)$.
\label{lem:conti}
\end{lemm}
We can easily see Proposition \ref{right_conti} from the lemma.
\begin{proof}[Proof of Theorem \ref{right_conti}]
From Remark \ref{rem:right_conti}, it is sufficient to show that $\phole(p)$ is right continuous at $p \in (0, 1-\bond)$. \par
Take $p_0 \in (0, 1-\bond)$. Clearly, $\phole(p) = \lim_{n \to \infty} P_p(0^* \connect \partial \B(n))$. The function $P_p(0^* \connect \partial \B(n))$ is continuous at $p_0$ from Lemma \ref{lem:conti}, and non-increasing with $n$. Thus $\phole(p)$ is upper semi-continuous. Since $\phole(p)$ is non-decreasing with $p$, thus $\phole(p)$ is right continuous at $p_0$.
\end{proof}
\begin{rem}
\label{bond_inbox}
\rm
For the case of bond percolation, it is easy to see that $P_p(0 \bondconnect \partial B(n))$ is continuous. Indeed, since the event $\{ 0 \bondconnect \partial B(n)\}$ depends only on the configuration of bonds in $B(n)$, $P_p(0 \bondconnect \partial B(n))$ is polynomial with $p$. Therefore, Proposition \ref{right_conti_bond} can be shown as above.
\end{rem}
Let us turn to the proof of Lemma \ref{lem:conti}. Unlike the case in Remark \ref{bond_inbox}, the event $\{ 0^* \connect \partial \B(n) \}$ depends on outside of any fixed box because a hole can be constructed through the outside. Here, we use Proposition \ref{exp_decay} (\cite[Theorem 8.21]{Grimmett}), which states that finite bond clusters in the supercritical phase are less likely to be large. In the context of hole percolation, this proposition states that holes cannot be large. We use this fact and approximate the event $\{ 0^* \connect \partial \B(n) \}$ by some local events.
\begin{prop}
\label{exp_decay}
\rm
Let $(\Omega, \F, P_p)$ be the probability space for bond percolation model with probability $p$.  Let $G_n$ be an event defined by
\begin{align*}
G_n := \{ 0 \bondconnect H(n)\mbox{ and } |C(0)| < \infty \}, 
\end{align*}
where $H(n) = \{x \in \mathbb{Z}^d : x_1 = n\}$. Then, for $\bond < p < 1$, there exists $\gamma(p)>0$ such that
\begin{align*}
P_p(G_n) \leq e^{-\gamma(p)n}.
\end{align*}
Moreover, we can take $\gamma:(\bond, 1) \longrightarrow \mathbb{R}_{>0}$ satisfying the condition that
\begin{align*}
\inf_{p \in [\alpha, \beta]} \gamma(p) > 0
\end{align*}
for any $\bond < \alpha < p < \beta < 1$.
\end{prop}
From this proposition, we will prove Lemma \ref{lem:conti}.
\begin{proof}[Proof of Lemma \ref{lem:conti}]
Fix arbitrary $p \in (0,1- \bond)$ and take $\bond < \alpha < 1- p < \beta <1$. We set the following events:
\begin{align*}
&A := \{0^* \connect \partial \B(n)\},	\\	
&A_k := \{ 0^* \connect \partial \B(n) \mbox{ even if all faces in } {\mathbb{R}}^d \setminus \Lambda^k \mbox{ are set to be closed}  \}.
\end{align*}
Then $P_p(A_k)$ is a polynomial with respect to $p$ since $A_k$ depends only on the configuration of faces in $\Lambda^k$, and thus continuous. Moreover, from 
\begin{align*}
A_1 \subset A_2 \subset A_3 \cdots \subset \bigcup_k A_k = A
\end{align*}
and the continuity of measures, we have $P_p(A_k) \longrightarrow P_p(A)$ as $k \longrightarrow \infty$.\par
We show that this convergence is uniformly on $1-p \in [\alpha, \beta]$, which leads to the continuity of $P_p(A)$ at $p$. Clearly, we can see
\begin{align*}
P_p(A) - P_p(A_k) &= P_p(A \setminus A_k).
\end{align*}
If the event $A \setminus A_k$ occurs, then in the dual lattice, there exists some $x^* \in \partial \B(n)$ which belongs to a finite dual cluster intersecting $\partial \B(k)$. Thus, we obtain
\begin{align}
\label{dual_finite}
P_p(A \setminus A_k) \leq \sum_{x^* \in \partial \B(n)} P_p( x^* \overset{\rm dual}{\longleftrightarrow} \partial \B(k), |C(x^*)| < \infty).
\end{align}
Here, $x^* \overset{\rm dual}{\longleftrightarrow} \partial \B(k)$ implies that there exists an open dual path connecting between $x^*$ and some dual vertex $y^*$ with $\|x^* - y^*\|_{\infty} \geq k-n$. Thus, the right hand side of (\ref{dual_finite}) is bounded above by
\begin{align}
\label{meidaihyouka}
|\partial \B(n)| P_p( 0^* \overset{\rm dual}{\longleftrightarrow} \partial \B(k - n), |C(0^*)| < \infty)
\end{align}
Let $F_1, \ldots , F_{2d}$ be the list of $2d$ surfaces of $\B(k-n)$. Then (\ref{meidaihyouka}) is again bounded above by
\begin{align*}
          &|\partial \B(n)| P_p( \bigcup_{i=1}^{2d} \{0^* \overset{\rm dual}{\longleftrightarrow} F_i \} \cap \{|C(0^*)| < \infty \}) \\
\leq \, &2d|\partial \B(n)| P_p(0^* \overset{\rm dual}{\longleftrightarrow} F_1, |C(0^*)| < \infty ) \\
\leq \, & 2d|\partial \B(n)|  e^{-\sigma (k-n)}, 
\end{align*}
where $\sigma := \inf_{1-p \in [\alpha, \beta]} \gamma(1-p)> 0$. This completes the proof of uniform convergence $P_p(A) \longrightarrow P_p(A_k)$ for $1-p \in [\alpha, \beta]$.
\end{proof}

\subsection{Left continuity of $\phole(p)$ in the supercritical phase}
For the proof of Proposition $\ref{left_conti_bond}$, the uniqueness of infinite clusters (Theorem $\ref{th:uniqueness}$) plays an important role. In this paper, we obtain the analogue theorem (Theorem \ref{unique}) in the previous section．We will prove the following Proposition by using the similar idea of the proof of Proposition $\ref{left_conti_bond}$. 
\begin{prop}
\rm
\label{left_conti}
$\phole(p)$ is left continuous on the interval $(\hole,1]$. 
\end{prop}
Note that from Lemma \ref{obs:taiou} and Theorem \ref{unique}, the number of infinite connected components in $(\mathbb{L}^d)^* - I$ is also at most 1. Remark that we focus on the infinite cluster $I$, and  ignore the configuration of dual bonds in $(\mathbb{L}^d)^* - I$. \par
In Subsection \ref{subsec:Face_percolation}, we adopt the probability space $(\Omega, \mathcal{F}, P_p)$ of the face percolation model for a fixed probability $p$. For the proof of Proposition \ref{left_conti}, however, we need to compare the configurations in different probability. Now we use the technique of ``coupling" of random variables in order to consider configurations in different probability on the same probability space. \par
\begin{proof} 
We set the collection of independent random variables $(X_Q: Q \in \mathcal K_{d-1}^d )$ on a probability space $(\tilde{\Omega}, \tilde{\mathcal{F}}, \tilde{P})$ indexed by $Q \in \mathcal K_{d-1}^d$, each having the uniform distribution on $[0, 1]$. For $p \in [0,1]$, we define $\eta_p \in \Omega=\{0,1\}^{\mathcal K^d_{d-1}} $ as 
\begin{align*}
  \eta_p(Q)= \begin{cases}
     1, & \text{if }X_Q < p, \\
     0, & \text{otherwise},
   \end{cases}
\end{align*}
which gives the face percolation model with probability $p$. Let $I^*(\eta_p) \subset (\mathbb{L}^d)^*$ be the infinite cluster in $ (\mathbb{L}^d)^*$ induced by $\eta_p$ and we set $H_p:= (\mathbb{L}^d)^*- I^*(\eta_p)$. We denote by $H_p(x^*)$ the connected component in $H_p$ including $x^*$. From Lemma \ref{obs:taiou}, we obtain the following relation:
\begin{align*}
&|H_p(0^*)| = \infty  \Longleftrightarrow |G_{0^*}| = \infty, \\
&\exists x^*  \in (\Z^d)^* \ \text{such that} \  |H_p(x^*)| = \infty  \Longleftrightarrow G \mbox{ has an infinite cluster}.
\end{align*}
Remember that $H_p$ and $H_p(0^*)$ do not have anything to do with whether or not the dual bond in $(\mathbb{L}^d)^*- I^*(\eta_p)$ is open. Note also that, from the definition of $\eta_p$, $H_p$ and $H_p(x^*)$ is increasing with $p$. \par
Under this setting, we have
\begin{align*}
\phole(p) = P_p(|G_{0^*}| = \infty) = \tilde{P}(|H_p(0^*)| = \infty).
\end{align*}
Take arbitrary $p > \hole$. Proposition \ref{left_conti} is equivalent to the equality
\begin{align}
\label{eq:leftconti}
\lim_{\pi \nearrow p} \tilde{P}(|H_{\pi}(0^*)| = \infty) = \tilde{P}(|H_p(0^*)| = \infty).
\end{align}
From the continuity of measures, the left hand side of (\ref{eq:leftconti}) is equal to 
\begin{align*}
\tilde{P}(\exists \pi < p \ \text{such that} \ |H_{\pi}(0^*)| = \infty ).
\end{align*}
From the monotonicity of $H_p(0^*)$, we can easily see $\{ \exists \pi < p \ \text{such that} \ |H_{\pi}(0^*)| = \infty \} \subset \{|H_p(0^*)| = \infty \}$. Thus, it is sufficient to show
\begin{align}
\label{eq:leftconti_2}
\tilde{P}(\{|H_p(0^*)| = \infty \} \setminus \{ \exists \pi < p \ \text{such that} \ |H_{\pi}(0^*)| = \infty \}) = 0.
\end{align}
Suppose $|H_p(0^*)| = \infty$.  Let us take $\alpha$ with $\hole < \alpha < p$. From Theorem \ref{unique}, there exists a dual vertex $x^*$ such that $|H_{\alpha}(x^*)| = \infty$ almost surely. Theorem \ref{unique} also implies $H_{\alpha}(x^*) \subset H_p(0^*)$ almost surely. (If not, for this $p$, there exist two infinite connected components,  $H_p(0^*)$ and the one including $H_{\alpha}(x^*)$. This contradicts the uniqueness of the infinite hole cluster.)
Thus, we can take a dual path $0^*=x_0^*,x_1^*,x_2^*, \ldots, x_n^*=x^*$ in $H_p(0^*)$. For each $i = 0, \ldots ,n-1$, we can take the faces such that $\eta_p(Q) = 1$ which construct the hole including $x_i^*$, since $x_i^* \in H_p(0^*)$. We denote these faces by $Q_{(i,1)}, \cdots ,Q_{(i,m_i)}$. Let $\pi$ be
\begin{align*}
\pi := \max \{ \alpha, X_{Q_{(i,j)} }: i = 0, \ldots ,n-1,  \  j = 1, \ldots ,m_i \}.
\end{align*}
Then $\alpha \leq \pi < p$ since  $X_{Q_{i,j}} < p$. For this $\pi$, we have $H_{\alpha}(x^*) \subset H_{\pi}(x^*)$ and $H_{\pi}(x^*)$ includes $0^*$. Thus we have $|H_{\pi}(0^*)| = \infty$. This completes the proof of (\ref{eq:leftconti_2}).
\end{proof}

\subsection{The number of vertices in the hole graph}
In this subsection, we study the number of vertices in the hole graph. From Remark \ref{Betti}, the number of vertices in the hole graph $G^n$ restricted to $\Lambda^n$ is equal to the Betti number $\beta^n(X)$ of $X \cap \Lambda^n$ in dimension $d-1$. It is shown by the paper \cite{tsunoda} that $\beta^n(X)/|\B(n)|$ converges to a certain constant as $n \longrightarrow \infty$. Here, we give the explicit value of this limit by using the function $\kappa(p) := E_p(|C(0)|^{-1})$, a basic function in percolation theory.
\begin{prop}
\rm
Let $|G^n(\omega)|$ be the number of vertices in the restricted hole graph $G^n(\omega)$. Then,
\begin{align*}
\lim_{n \to \infty} \frac{|G^n(\omega)|}{|\B(n)|} = \kappa(1-p),
\end{align*}
almost surely.
\label{prop:vertices}
\end{prop}
\begin{proof}
Note that under the correspondence between dual bonds and faces, $|G^n|$ is equal to the number of finite clusters of the dual lattice whose vertices are all in $\B(n)$. For $x^* \in (\Z^d)^*$, let us define the two random variables $f_{x^*}$, $g_{x^*}$ as follows:
\begin{align*}
f_{x^*}(\omega) &= |C^*(x^*)|^{-1}, \\
g_{x^*}(\omega) &= \left \{
\begin{array}{cc}
|C^*(x^*)|^{-1}, & \mbox{if } C^*(x^*) \subset \B(n), \\
0, & \mbox{otherwise},
\end{array}
\right. 
\end{align*}
respectively. Clearly, we see
\begin{align*}
\sum_{x^* \in \B(n)} g_{x^*}(\omega) = |G^n(\omega)| 
\end{align*}
and from the ergodic theorem (\cite[Proposition 2.2]{ergodic}), we obtain almost surely
\begin{align*}
\lim_{n \to \infty} \frac{1}{|\B(n)|} \sum_{x^* \in \B(n)} f_{x^*}(\omega) = \kappa(1-p) 
\end{align*}
We show that the same equality holds even if we replace $f_{x^*}$ by $g_{x^*}$. Since
\begin{align*}
|f_{x^*}(\omega) - g_{x^*}(\omega)| &= \left \{
\begin{array}{cc}
|C^*(x^*)|^{-1}, & \mbox{if } x^*  \bondconnect  \partial \B(n+1), \\
0, & \mbox{otherwise}, 
\end{array}
\right. 
\end{align*}
we obtain
\begin{align*}
\left|\sum_{x^* \in \B(n)} f_{x^*}(\omega) - \sum_{x^* \in \B(n)} g_{x^*}(\omega) \right| 
&\leq \sum_{x^* \in \B(n)} | f_{x^*}(\omega) - g_{x^*}(\omega) | \\
&\leq \sum_{x^* \in \B(n): \, x^*  \bondconnect  \partial \B(n+1) } |C^*(x^*)|^{-1} \\
&\leq \sum_{x^*  \bondconnect  \partial \B(n+1) } |C^*(x^*)|^{-1}.
\end{align*}
The last summand is equal to the number of cluster intersecting $\partial \B(n+1)$, and thus bounded above by $|\partial \B(n+1)|$. Therefore, 
\begin{align*}
       \frac{1}{|\B(n)|} \left|\sum_{x^* \in \B(n)} f_{x^*}(\omega) - \sum_{x^* \in \B(n)} g_{x^*}(\omega) \right|
      & \leq \frac {|\partial \B(n+1)|}{|\B(n)|} \ \longrightarrow 0
\end{align*}
as $n \longrightarrow \infty$. This completes the proof of Proposition \ref{prop:vertices}.
\end{proof}

\subsection{The size of holes}
Though a hole graph is constructed from elementary cubes, its structure may be much more complicated than a subgraph of $(\mathbb{L}^d)^*$. Indeed, for example, the degree of the hole graph can be unbounded since the ``size" of holes are unbounded. In this subsection, as a first step in studying the structure of a hole graph, we give a proposition about the average size of holes. Here, we define the {\it size} $\sz(D)$ of a hole $D$ as the number of dual vertices in $D$, i.e., 
\begin{align*}
\sz(D) := \# \{x^* \in (\Z^d)^* : x^* \in D\}.
\end{align*}
\begin{prop}
\rm
\begin{align*}
\lim_{n \to \infty} \frac{1}{|G^n(\omega)|}  \sum_{D \in G^n(\omega)} \sz(D) = \frac{1 - \pbond(1-p)}{\kappa(1-p)},
\end{align*}
almost surely.
\label{prop:average}
\end{prop}
Note that the left hand side represents the average of size of holes in $G(\omega)$.
\begin{proof}
It suffices to prove that
\begin{align*}
\lim_{n \to \infty} \frac{1}{|\B(n)|}  \sum_{D \in G^n(\omega)} \sz(D) = 1 - \pbond(1-p),
\end{align*}
almost surely. Indeed, together with Proposition \ref{prop:vertices}, we obtain almost surely
\begin{align*}
\frac{1}{|G^n(\omega)|}  \sum_{D \in G^n(\omega)} \sz(D) 
&= \left(\frac{1}{|\B(n)|}  \sum_{D \in G^n(\omega)} \sz(D) \right){\left(\frac{1}{|\B(n)|} |G^n(\omega)| \right)}^{-1} \\
&\longrightarrow \frac{1 - \pbond(1-p)}{\kappa(1-p)}
\end{align*}
as $n \longrightarrow \infty$. For the indicator function $I_{\{C(x^*) \subset \B(n)\}}$, we easily see 
\begin{align*}
\sum_{x^* \in \B(n)} I_{\{C(x^*) \subset \B(n) \}} = \sum_{D \in G^n(\omega)} \sz(D).
\end{align*}
Moreover, from the ergodic theorem (\cite[Proposition 2.2]{ergodic}), we obtain almost surely
\begin{align*}
\lim_{n \to \infty} \frac{1}{|\B(n)|} \sum_{x^* \in \B(n)} I_{\{|C(x^*)| < \infty \}} = 1-\pbond(1-p).
\end{align*}
Thus, similar to Proposition \ref{prop:vertices}, let us show that the same equality holds even if we replace $I_{\{|C(x^*)| < \infty \}}$ by $I_{\{C(x^*) \subset \B(n)\}}$. It suffices to show 
\begin{align}
\label{eq:I-I=0}
\lim_{n \to \infty} \frac{1}{|\B(n)|} \left|\sum_{x^* \in \B(n)} I_{\{C(x^*) \subset \B(n)\}} - \sum_{x^* \in \B(n)} I_{\{|C(x^*)| < \infty \}} \right| = 0. 
\end{align}
Fix arbitrary $\epsilon >0$. Since
\begin{align*}
\lim_{l \to \infty} P_p(l \leq C(0^*) < \infty) = P_p(\bigcap_{l=1}^{\infty} \{l \leq C(0^*) < \infty\}) = 0,
\end{align*}
we may take sufficiently large $l \in \mathbb{N}$ such that
\begin{align*}
P_p(l \leq C(0^*) < \infty) < \epsilon.
\end{align*}
We fix this $l$. Clearly, we have
\begin{align*}
    &\frac{1}{|\B(n)|} \left|\sum_{x^* \in \B(n)} I_{\{C(x^*) \subset \B(n)\}} - \sum_{x^* \in \B(n)} I_{\{|C(x^*)| < \infty \}} \right| \\
 = &\frac{1}{|\B(n)|} \#\{x^* \in \B(n) : |C(x^*)| < \infty \mbox{ and }x^*  \bondconnect  \partial \B(n+1) \}.
\end{align*}
The right hand side can be decomposed as
\begin{align}
\label{eq:division}
&\frac{1}{|\B(n)|} \#\{x^* \in \B(n-l) : |C(x^*)| < \infty \mbox{ and }x^*  \bondconnect  \partial \B(n+1) \} \notag \\
  &+\frac{1}{|\B(n)|} \#\{x^* \in \B(n) \setminus \B(n-l) : |C(x^*)| < \infty \mbox{ and }x^*  \bondconnect  \partial \B(n+1) \}.
\end{align}
The first term of (\ref{eq:division}) is bounded above by
\begin{align*}
\frac{1}{|\B(n)|} \#\{x^* \in \B(n) : l \leq C(x^*) < \infty \}.
\end{align*}
From the ergodic theorem (\cite[Proposition 2.2]{ergodic}), it converges to $P_p(l \leq C(0^*) < \infty)$ as $n \longrightarrow \infty$ almost surely.  Thus, for sufficiently large $n$, the first term is bounded above by
\begin{align*}
P_p(l \leq C(0^*) < \infty) + \epsilon < 2\epsilon.
\end{align*}
The second term of (\ref{eq:division}) is bounded above by
\begin{align*}
\frac{1}{|\B(n)|} |\B(n) \setminus \B(n-l)| \longrightarrow 0
\end{align*}
as $n \longrightarrow \infty$. This completes the proof of (\ref{eq:I-I=0}).
\end{proof}

\section{Conclusions}
In this paper, we introduced the hole percolation model, and  gave the estimates for the critical probability of this model. Moreover, we proved the uniqueness of the infinite hole cluster and showed the estimate of the connectivity probability $P_p(x^* \connect y^*)$. 
Then, in view of the classical percolation models, the following problems will be important to obtain further understandings of the hole percolation.

\begin{itemize}
\item The estimate of the convergence velocity $P(x^* \connect y^*) \longrightarrow 0$ as $\|x^* - y^*\|_1 \longrightarrow \infty$ in the subcritical phase is yet to be given, while for the classical bond percolation model, the exponential decay (\ref{subcritical}) is shown. 
\item  It should be clarified whether the inequality $\hole(d) \leq 1- \bond(d)$ in Theorem \ref{theo:pc_estimate} is strict $\hole(d) < 1- \bond(d)$ or not. If it is strict, this implies that there exist both an infinite hole cluster and an infinite dual bond cluster almost surely for $\hole(d) < p < 1- \bond(d)$. This means the existence of an infinite hole cluster which does not cover the whole $\R^d$. By combining with Lemma \ref{obs:taiou} and \cite[Theorem 1.1]{finite_cluster}, we obtain this strictness for $d \geq 19$. Moreover, with the work of Fitzner-van der Hofstad \cite[Theorem 1.6]{greaterthan11}, it can be extended to $d \geq 11$. Yet it has not been shown for $3 \leq d \leq 10$. 

\end{itemize}

It should also be remarked that our hole percolation model, which is introduced as a higher dimensional percolation model, is limited to holes defined by homology generators in codimension one. From the viewpoint of theoretical generality, it is desirable to introduce other types of percolation models which can also deal with clusters of holes defined by homology generators in arbitrary codimension. However, one of the difficulties of the strategy introduced in this paper is that there is no canonical correspondence between the ``$k$-dimensional holes" of a random cubical set and the homology generators  in dimension $k$ except for $k=d-1$. 

For example, let us consider the bond percolation model in $\mathbb{L}^3$ and focus on the $1$-dimensional holes (i.e., loops) in the random graph. Then, it is easily observed that there is no natural bijective correspondence between loops and homology generators in dimension one. Figure \ref{fig:generators} shows the $1$-dimensional skeleton $X$ of the $3$-dimensional unit cube. Although $\text{rank} \, H_1(X) = 5$, 
there is no natural choice of 5 representative loops in $X$.
\begin{figure}[H] 
   \centering
   \includegraphics[width=3cm]{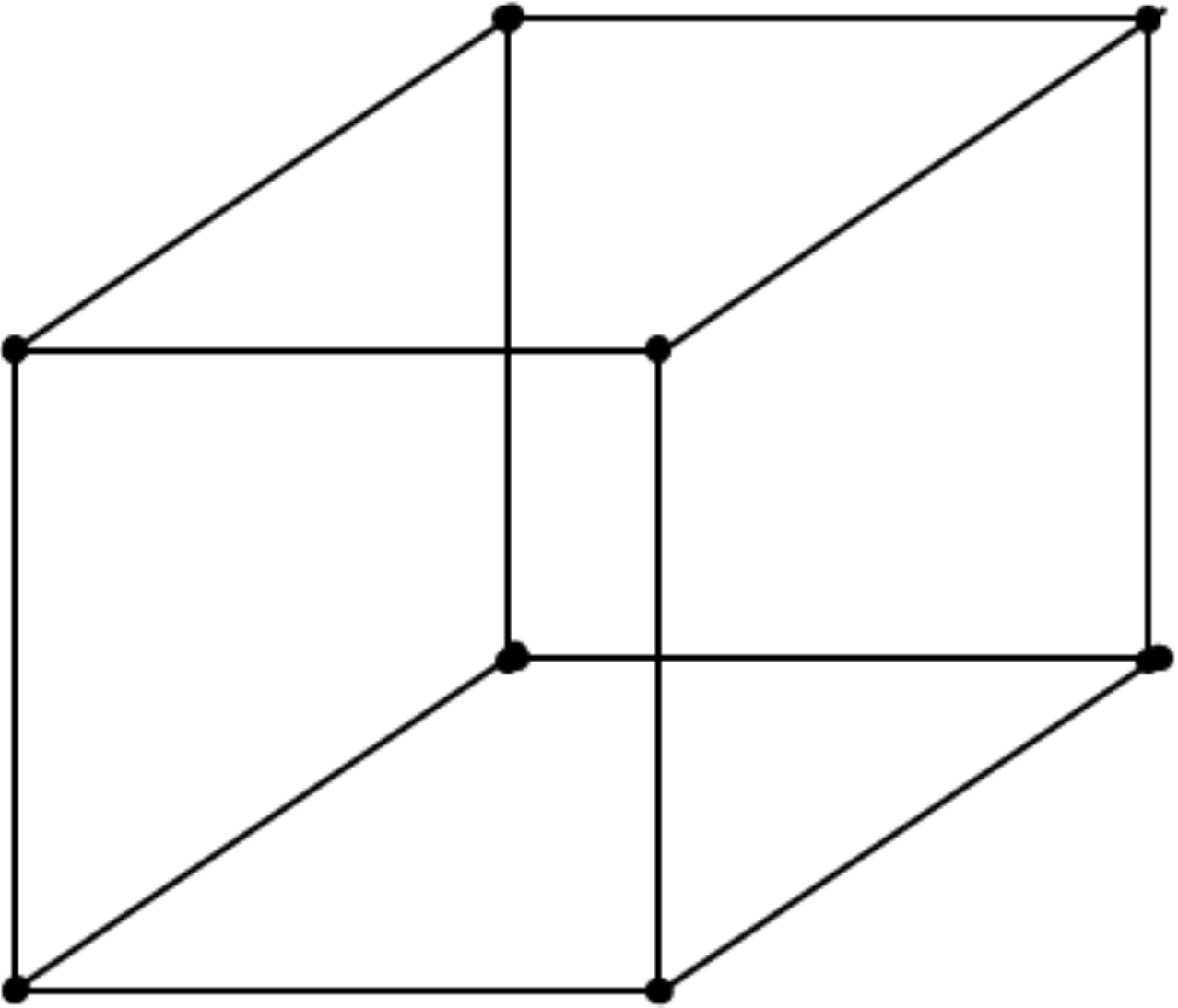} 
   \caption{The 1-skeleton $X$ of the unit cube. }
   \label{fig:generators}
\end{figure}

\paragraph{Acknowledgements}
The authors would like to thank Tomoyuki Shirai, Kenkichi Tsunoda and Masato Takei for their valuable suggestions and useful discussions. This work is partially supported by JST CREST Mathematics 15656429.

\end{document}